\documentclass[12pt]{article}
\usepackage{amsmath,amssymb,amsthm,amsfonts,amscd,euscript,verbatim, t1enc, newlfont, xypic}
\usepackage{hyperref}
\hypersetup{breaklinks}

 \theoremstyle{plain}
\newtheorem{theo}{Theorem}[subsection]

\newtheorem{pr}[theo]{Proposition}

 \newtheorem{lem}[theo]{Lemma}
 \newtheorem{coro}[theo]{Corollary}
  
\theoremstyle{remark}
\newtheorem{rema}[theo]{Remark}

\theoremstyle{definition}
\newtheorem{defi}[theo]{Definition}

 \newcommand\lan{\langle}
\newcommand\ra{\rangle}

\newcommand\gd{\mathfrak{D}}
\newcommand\gdp{\mathfrak{D}'}

\newcommand\ob{^{-1}}

\newcommand\dmge{DM^{eff}_{gm}{}}

\newcommand\dme{DM_-^{eff}{}}
\newcommand\dm{DM}

\newcommand\cha{\operatorname{char}}
\newcommand\proo{\operatorname{Pro}}

\newcommand\mg{M_{gm}}

\newcommand\obj{Obj}

\newcommand\id{id}
\newcommand\cu{\underline{C}}
\newcommand\du{\underline{D}}
\newcommand\au{\underline{A}}

\newcommand\cupr{\underline{C}'}
\newcommand\cuperp{\underline{C}_{\perp}}

\newcommand\com{\mathbb{C}}

\newcommand\z{{\mathbb{Z}}}
\newcommand\q{{\mathbb{Q}}}
\newcommand\af{\mathbb{A}}
 
\newcommand\p{\mathbb{P}}

\newcommand\pt{\operatorname{pt}}

\newcommand\al{\alpha}
\newcommand\be{\beta}
\newcommand\si{\sigma}

\newcommand\ns{\{0\}}
\DeclareMathOperator\homm{\operatorname{Hom}}
\DeclareMathOperator\prli{\varprojlim}
\DeclareMathOperator\inli{\varinjlim}

\DeclareMathOperator\codim{\operatorname{codim}}
\newcommand\chow{Chow}

\newcommand\ab{\underline{Ab}}
\newcommand\var{Var}
\newcommand\sv{SmVar}
\newcommand\spv{SmPrVar}

\newcommand\spe{\operatorname{Spec}}

\newcommand\modd{\operatorname{Mod}}

\newcommand\brj{\lan j\ra}
\newcommand\brjj{\{ j\}}
\newcommand\gmm{\mathbb{G}_m}
\newcommand\gmmpl{\mathbb{G}_{m+}}

\newcommand\zoh{\z[2\ob]}
\newcommand\sh{SH^{S^1}(k)}
\newcommand\shc{SH^{S^1}(k)^c}

\newcommand\psh{PSH^{S^1}(k)}
\newcommand\sht{SH(k)}
\newcommand\shtc{SH(k)^c}
\newcommand\shmgl{SH^{\mgl}(k)}
\newcommand\shctpl{SH^{+}(k)^c}

\newcommand\psht{PSH^{T}(k)}

\newcommand\shtoh{SH(k)[2\ob]}

\newcommand\gdb{\gd^{big}}

\newcommand\gdt{\gd^T}

\newcommand\gdtpl{\gd^{+}}

\newcommand\prpl{pr^+}

\newcommand\prplgd{pr^+_{\gdt}}

\newcommand\shtpl{SH^+(k)}
\newcommand\shtmi{SH^-(k)}

\newcommand\sinf{\Sigma^{\infty}}
\newcommand\sinft{\Sigma_T^{\infty}}

\newcommand\om{\sinf}
\newcommand\omt{\sinft}
\newcommand\omtpl{\Sigma_{T+}^{\infty}}
\newcommand\pom{P\sinf}
\newcommand\pomt{P\sinft}

\newcommand\sm{\underline{SmVar}}
\newcommand\opa{\mathcal{OP}}
\newcommand\popa{\proo-\mathcal{OP}}
\newcommand\afo{\mathbb{A}^1}
\newcommand\dopsh{\Delta^{op}Pre_{\bullet}}
\newcommand\dosh{\Delta^{op}Shv_{{\bullet}}}
\newcommand\doshp{\Delta^{op}Shv_{{\bullet}}'}

\newcommand\dopshs{H_s\Delta^{op}Pre_{\bullet}}
\newcommand\doshs{H_s\Delta^{op}Shv_{{\bullet}}}

\newcommand\hk{\mathcal{H}_k}

\newcommand\pspt{Pre_{\bullet}}
\newcommand\spt{Shv_{{\bullet}}}
\newcommand\ho{\operatorname{Ho}}
\newcommand\hogd{\ho_\gd}

\newcommand\hosh{\ho_{\sh}}
\newcommand\mz{MZ}
\newcommand\mgl{MGl}
\newcommand\mglmod{MGl-\modd(k)}
\newcommand\gdmgl{\gd^{\mgl}}
\newcommand\gdm{\gd^{mot}}
\newcommand\tcho{t_{Chow}}

 \DeclareMathOperator\ke{\operatorname{Ker}}
 
\DeclareMathOperator\imm{\operatorname{Im}}
\DeclareMathOperator\co{\operatorname{Cone}}

\DeclareMathOperator\kar{\operatorname{Kar}}
\DeclareMathOperator\adfu{\operatorname{AddFun}}

\newcommand\hrt{{\underline{Ht}}}
\newcommand\hrtt{{\underline{Ht}}^T}
\newcommand\hrtpl{{\underline{Ht}}^+}

\newcommand\hw{{\underline{Hw}}}

\begin{document}

 \title{Gersten weight structures for motivic homotopy categories; direct summands of cohomology of function fields,  and 
 coniveau spectral sequences } 
 \author{M.V. Bondarko
   \thanks{ 
 The work is supported by RFBR
(grants no. 14-01-00393-a and 15-01-03034-a), by  Dmitry Zimin's Foundation "Dynasty", and by the Scientific schools grant no. 3856.2014.1.}}\maketitle

\begin{abstract}
In this paper for any cohomology theory $H$ that can  be
factored through (the Morel-Voevodsky's triangulated motivic homotopy category) $\sh$
we establish  the $\sh$-functoriality of  coniveau spectral sequences for $H$.
We also prove the following interesting result: for any affine essentially smooth semi-local $S$ the Cousin complex for $H^*(S)$ splits; if $H$ also factors through $\shtpl$ or $\shmgl$, then this is also true for any primitive $S$. 
Moreover, the cohomology of such an $S$ is a direct summand of the cohomology of any its open dense subscheme.

In order to prove these facts we construct 
a triangulated category $\gd$ of {\it motivic pro-spectra} that contains $\shc$ as well as certain 
pro-spectra of function fields over $k$ (along with its $\sht$-version $\gdt$, and certain $\gdtpl$ and $\gdmgl$). We decompose the  $\sh$-spectrum of a smooth
variety
 (in the sense of Postnikov towers) into twisted (pro)spectra of its
points. 
 Those belong to the heart of a {\it Gersten} weight structure $w$ on $\gd$, whereas 
  {\it weight spectral
  sequences} corresponding to $w$ generalize the 
 'classical' coniveau ones
  (to the cohomology of arbitrary objects of $\shc$). When a cohomological functor
   is represented
 by a $Y\in \operatorname{Obj} \sh$,  the corresponding coniveau spectral
 sequences can be expressed in  terms of the (homotopy)
$t$-truncations of $Y$; this extends to motivic spectra the seminal coniveau
spectral sequence computations of Bloch and Ogus.
We also prove several (other) direct summand results. In particular,  the pro-spectra of function
fields (in $\gd$ or $\gdt$) contain twisted pro-spectra of their residue fields (for all
geometric valuations); hence the same is true for any cohomology of these fields. 

The analogues of these results for Voevodsky's motives (i.e., for $\dme(k)$ instead of $\sh$) over countable fields were proved by the author in a previous paper. In the current paper we use model categories (instead of differential graded ones) for the construction of the corresponding $\gd$, and apply a new weight structure construction result (for cocompactly cogenerated triangulated categories). 
We also prove a certain '$\shtpl$-acyclity' statement for primitive schemes; this result could be interesting in itself. 


\end{abstract}

\tableofcontents

 \section*{Introduction}

 In  \cite{bger}  for a countable perfect field $k$  and any cohomology theory $H$ that factors through the Voevodsky's  $\dmge(k)$ the following results were established: 
 coniveau spectral sequences (for the cohomology of smooth varieties) can be functorially extended to $\dmge(k)$;  the cohomology of an essentially smooth affine semi-local scheme  $S/k$  
 (and more generally, of a {\it primitive pro-scheme})
   is a retract of the cohomology of any its pro-open dense subscheme. 
  Note that both of these results are far from being obvious; the second one yields 
  that the augmented Cousin complex for the cohomology of an (affine essentially smooth)  semi-local (or primitive) $S$ splits. 
   Also, for an $H$ represented by an object $C$ of $\dme(k)$ (a {\it motivic complex}) it was proved that  coniveau spectral sequences for $H$
 can be 'motivically functorially' expressed in terms of the $t$-truncations of $C$ (with respect to the homotopy $t$-structure on $\dme(k)$); this is an extension of an important result of \cite{blog}.
 
The goal of the current paper is to extend these results to arbitrary perfect base fields and to a much wider class of cohomology theories. 
This class includes $K$-theory of various sorts, algebraic, semi-topological, and complex cobordism, and Balmer's Witt cohomology.
We succeed in proving these results via considering  certain triangulated categories $\gd\supset \shc$ and $\gdt\supset \sht^c$ (as well as $\gdtpl\supset \shtpl^c$ and $\gd^{\mgl}\supset \shmgl{}^c =\ho(\mglmod)^c$), and introducing certain {\it Gersten weight structures} on them. These weight structures are {\it orthogonal} to the homotopy $t$-structures for the corresponding motivic categories (via
certain {\it nice dualities} that we construct). 
The formalism of weight structures easily yields 
 several functoriality and direct summand results (cf. Theorem \ref{tshtt}(II\ref{ids1}--\ref{isplc}, III) below).
 Note that these statements have a variety of versions,  depending on the choice of a 'motivic' category through which a given cohomology theory $H$ factors. We discuss this matter (along with several 'concrete' examples of  cohomology theories) in \S\ref{sexamp}; since all oriented cohomology theories factor through $\shmgl$, one can apply the strongest versions of our results to them.
   In any case, our direct summand results are much stronger than the {\it universal exactness} statement
 of \cite{suger} (see Remark \ref{runiv} for more details). Besides, though the setting of ibid. seems to be somewhat more general, the author does not know of any concrete example of a cohomology theory that satisfies the conditions of ibid. and is not representable by an object of $\sh$.

We also note that our consideration of (triangulated)  pro-spectral categories (together with the purity distinguished triangle for {\it pro-schemes}; see Proposition \ref{pinfgy}) can be interesting in itself independently from the rest of the paper. It seems that the only triangulated motivic categories of pro-objects considered in the existing literature are the {\it comotivic} one 
of \cite{bger} and the pro-spectral motivic categories studied  in the current paper; note in contrast that the category of {\it pro-motives} introduced in \S3.1 of \cite{deggenmot} is not triangulated. Our methods of constructing motivic pro-spectra are far from being original (we mostly use the results of \cite{tmodel} and related papers); yet this has 
the following advantage: they certainly allow the construction of motivic pro-spectra over an arbitrary base scheme $S$.

Another interesting result is our 
Proposition \ref{pshtt}(\ref{ipshinvp}) (on the '$\tau$-positive spectral acyclity' of primitive schemes).

Besides, in the current paper we
study (in detail) 
 weight structures for cocompactly cogenerated triangulated categories (certainly, this can be dualized); this seems to be a rather important new piece of 
(more or less) abstract 
homological algebra. 
 Whereas the 
 (first) existence of a weight structure statement 
 in Theorem \ref{tnews} is not quite new (
 and was essentially proved earlier by D. Pauksztello; yet we state it in a more explicit form), several other results are. Note also that our current 
 methods 
 for working with these weight structures are much more interesting than their (motivic 'bounded') analogues in \cite{bger} (and so can be used in various settings); see \S\ref{sdm} below for more remarks on this matter. In particular, it seems that the dual of Theorem \ref{tnews} (on weight structures in a compactly generated triangulated category) can be useful for various applications.
Also, Proposition \ref{pextpure} describing {\it pure extended} cohomology theories is quite new. 

 We also mention the relation of our results to
 certain ones  of F. Deglise (see 
  \cite{ndegl} and \cite{degorient}). 
  The methods and results of the papers mentioned (seem to) allow the description of coniveau spectral sequences for $\sh$-representable 
 theories in terms of the homotopy $t$-structure on $\sh$ (see Remark 4.4.2 of \cite{degorient} for the $\sht$-version of this result). Yet the corresponding isomorphisms depend on 
 choices of certain 'enhancements' for cohomology; so they 
only yield the functoriality of these isomorphisms with respect to flat pullbacks and proper pushforwards (note that several previous results on this subject also have similar drawbacks). 
 Besides, the results of Deglise definitely do not yield our direct summands statements. 
  
Lastly, note that we prove (see Proposition \ref{rwss}(II.3) and Theorem \ref{tshtt}(II.\ref{itsht})) the following statement: any Noetherian subobject of any possible image of 
any {\it (extended}) cohomology of a compact motivic spectrum belonging to $\sh^{t\le -r}$ or to $\sht^{t^T\le -r}$ (i.e., of a spectrum that is $r-1$-connected in the sense of \cite{morintao}, for any $r>0$) in the cohomology of a smooth variety is supported in codimension $\ge r$. 


Let us  describe the contents  of the paper. Some more information of this sort can be found at the beginnings of sections (also we  list some of the main notation in the end of \S\ref{snot}).

We begin \S\ref{sra1} with some notation (for triangulated categories and related matters).  We also introduce {\it pro-schemes} (certain projective limits of pointed smooth varieties) and  recall certain properties of the $\af^1$-stable homotopy category, 
 of the homotopy $t$-structure on it,  
and list some 
related properties of the corresponding cohomology of 
 semi-local schemes.

In \S\ref{sws} we recall the formalism of weight structures, their relation to ({\it orthogonal}) $t$-structures, weight filtrations, and spectral sequences. We also prove 
several new results on 
weight structures in cocompactly cogenerated triangulated categories, and give a description of {\it pure} cohomology for triangulated categories endowed with weight structures. 

In \S\ref{scomot}
we embed $\sh$ into a certain  triangulated  category
$\gdb$ of  {\it pro-spectra}. In this section we only give an 'axiomatic' description of pro-spectra; 
its construction will be described in \S\ref{sconstr}. 
We use the  main properties of
$\gdb$ for the construction of a certain {\it Gersten
weight structure} $w$ on  $\gd\subset \gdb$ ($\gd$ is the subcategory of $\gdb$ cogenerated by $\shc$; hence it is the largest subcategory of $\gdb$ 'detected by the Gersten weight structure'). 
$w$ possesses several nice properties; in particular, the pro-spectra of function fields and (affine essentially smooth) semi-local schemes over $k$
  belong to  $\gd_{w= 0}$ (and 'cogenerate' it), whereas the pro-spectra of arbitrary pro-schemes belong to $\gd_{w\ge 0}$. We also easily prove the following: 
if $S$ is a semi-local scheme (and $k$ is infinite),  $S_0$ is its dense sub-pro-scheme, then $\om(S_+)$ is a direct summand of $\om(S_{0+})$;
$\om(\spe K_+)$ (for a function field $K/k$)
 contains (as retracts)
    the pro-spectra of semi-local schemes whose generic point
is $\spe K$,
    as well as the twisted pro-spectra of residue
 fields of $K$ (for all geometric valuations). 
Note that all our arguments easily carry over to the 'stable' setting (i.e., to $\sht$ instead of $\sh$); we discuss this matter in \S\ref{ssupl}.

\S\ref{sapcoh} is central for this paper.
We translate the results  of the previous section to cohomology; in particular, we prove that the 
augmented Cousin complex for the cohomology of a semi-local pro-scheme splits. We also establish 
that weight spectral sequences and filtrations corresponding to $w$ are canonically isomorphic to the 'usual' coniveau ones (for the cohomology of smooth varieties). On the other hand, the general theory of weight spectral sequences yields that the corresponding spectral sequence $T(H,M)$ converging to $H^*(M)$ (for $M$ being an object of $\shc$ or $\gd$) is $\shc$-functorial in $M$ starting from $E_2$; this is far from being trivial from the 'classical' definition of coniveau spectral sequences. Next we construct a {\it nice duality} $\Phi:\gd^{op}\times \sh\to \ab$ (one may say that this is a 'regularization' for the corresponding restriction of the bifunctor $\gdb(-,-)$). Since  $w$ is {\it orthogonal} to $t$ with respect to $\Phi$,  our (generalized) coniveau spectral sequences can be expressed in terms of $t$ (starting from $E_2$); this vastly generalizes the corresponding results of Bloch and Ogus. We 
also explain that  for certain varieties and spectra one can choose quite 'economical' 
versions of weight Postnikov towers and (hence) of  generalized coniveau spectral sequences for cohomology.
We conclude the section by a description of all {\it pure extended} cohomology theories in terms of $\hw$, and 
prove the equivalence of $\hrt$ with the category of contravariant additive functors $\hw\to\ab$ converting all products into coproducts. 

In \S\ref{sconstr}
we consider certain model categories and construct the categories $\gdp$ and $\gdb=\ho(\gdp)$ satisfying the properties listed in \S\ref{scomot}.

In \S\ref{ssupl} we describe 
possible variations of our methods and results. In particular, we prove the natural analogues  of our main results for the triangulated categories of ($\tau$-positive) $T$-spectra and $\mgl$-modules. It turns out that for the corresponding pro-spectral categories $\gdtpl$ and $\gdmgl$ the hearts of the corresponding Gersten weight structures contain the pro-spectra of {\it primitive} schemes; so we obtain 'more splitting results' than in the $\sh$-setting.  We also (briefly) compare our methods with the ones of \cite{bger}; the current methods yield that the results of ibid. are also valid in the case where $k$ is not (necessarily) countable.

Next we explain the consequences of our results for 'concrete' cohomology theories (algebraic, topological, and Hermitian $K$-theory, Balmer's Witt groups, singular, algebraic and complex cobordism, motivic and \'etale cohomology).
We finish the paper 
by a discussion of some other possible methods for constructing (certain modifications of) our pro-spectral categories.

Our notation below partially follows the one of  \cite{morintao}. 
 All 'concrete' $t$-structures considered in this paper are certain versions of (Morel's) homotopy ones (yet our general notation for $t$-structures is quite distinct from the one of ibid., and we introduce it in \S\ref{dtst}).

The author is deeply grateful to  prof.
M. Levine for his very useful remarks and also for his hospitality during the author's staying in the Essen University. 
He would also like to express his gratitude to prof. A. Ananyevskiy, prof. B. Chorny, prof. F. Deglise, prof. G. Garkusha, prof. A. Gorinov,  prof. F. Muro,  prof. I. Panin, prof. T. Porter, and prof. D. White.

\section{Preliminaries: $t$-structures, Postnikov towers, and $\af^1$-connectivity of spectra}\label{sra1}


In this section we recall some notation, introduce some new one, and state some facts on motivic homotopy categories.


In \S\ref{snot} we introduce some (mostly, categorical) notation.

In \S\ref{dtst} we
recall the notion of  $t$-structure
(and introduce some notation for it) and of a Postnikov tower for an object of a triangulated category. 

In \S\ref{sprs} we define the category $\opa$ (of open embeddings in $\sv$) and $  \popa$ (certain objects of the latter are called pro-schemes). 
 
In \S\ref{ssh}
 we recall some 
properties of $\sh$ and the homotopy $t$-structure on it.


In \S\ref{sprim} we define primitive (pro)schemes (we will need them in \S\ref{ssupl}).

\subsection{Notation and conventions}\label{snot} 

For categories $C,D$ we write 
$D\subset C$ if $D$ is a full 
subcategory of $C$.

 For a category $C,\ X,Y\in\obj C$, we denote by
$C(X,Y)$ the set of  $C$-morphisms from  $X$ to $Y$.
We will say that $X$ is  a {\it
retract} of $Y$ if $\id_X$ can be factored through $Y$. Note that if $C$ is a triangulated or abelian category
then $X$ is a  retract of
$Y$ if and only if $X$ is its direct summand. 

For any $D\subset C$
the subcategory $D$ is called {\it Karoubi-closed} in $C$ if it
contains all retracts of its objects in $C$. We will call the
smallest Karoubi-closed subcategory of $C$ containing $D$  the {\it
Karoubi-closure} of $D$ in $C$; sometimes we will use the same term
for the class of objects of the Karoubi-closure of a full subcategory
of $C$ (corresponding to some subclass of $\obj C$).

The {\it Karoubization} $\kar(B)$ (no lower index) of an additive
category $B$ is the category of ``formal images'' of idempotents in $B$
(so $B$ is embedded into an idempotent complete category; it is
triangulated if $B$ is).
We will say that $B$ is Karoubian if the canonical embedding $B \to \kar
(B)$ is an equivalence, i.e.,
if any idempotent morphism yields a direct sum decomposition in $B$.

For a category $C$ we denote by $C^{op}$ its opposite category.

 For an additive category $\cu$ an  $X\in \obj\cu$ is called cocompact if
$\cu(\prod_{i\in I} Y_i,X)=\bigoplus_{i\in I} \cu(Y_i,X)$ for any
set $I$ and any $Y_i\in\obj \cu$ (below we will only consider cocompact objects in categories closed with respect to arbitrary small products). 
Dually, a compact object of $\cu$ is a cocompact object of $\cu^{op}$, i.e., $M$ is compact (in a $\cu$  closed with respect to arbitrary small coproducts) if $\cu(M,-)$ respects coproducts.

For $X,Y\in \obj \cu$ we will write $X\perp Y$ if $\cu(X,Y)=\ns$.
For $D,E\subset \obj \cu$ we will write $D\perp E$ if $X\perp Y$
 for all $X\in D,\ Y\in E$.
For $D\subset \cu$ we will denote by $D^\perp$ the class
$$\{Y\in \obj \cu:\ X\perp Y\ \forall X\in D\}.$$
Sometimes we will denote by $D^\perp$ the corresponding
 full subcategory of $\cu$. Dually, ${}^\perp{}D$ is the class
$\{Y\in \obj \cu:\ Y\perp X\ \forall X\in D\}$. 

In this paper all complexes will be cohomological, i.e., the degree of
all differentials is $+1$; respectively, we will use cohomological
notation for their terms. We will need the following easy observation on the homotopy category of complexes for an arbitrary 
additive category $B$.

\begin{lem}\label{lrcomp}
Suppose 
 a complex $M=(M^i)\in \obj K(B)$ is a $K(B)$-retract of an object of $B$ (i.e., of a complex of the form $\dots 0\to 0\to M'\to 0\to 0\to \dots$ for some  $M'\in \obj B$; we denote it by $M'[0]$). Then $M$ is also a retract of $M^0[0]$.  

\end{lem}
\begin{proof}
Suppose that the factorization $M\stackrel{f}{\to} M'[0] \stackrel{g}{\to}M$ yields the $K(B)$-retraction mentioned. Then $h=g\circ f$ cannot have non-zero components in non-zero degrees, and we can consider it as a morphism $M\to M^0[0]$ and vice versa. Hence $h$  yields the retraction in question.
\end{proof}

$\cu$ and $\du$ will usually denote some triangulated categories.
 We will use the
term {\it exact functor} for a functor of triangulated categories (i.e.,
for a  functor that preserves the structures of triangulated
categories).

A class $D\subset \obj \cu$ will be called {\it extension-closed} if $0\in D$ and for any
distinguished triangle $A\to B\to C$  in $\cu$ we have the implication $A,C\in
D\implies B\in D$. In particular, an extension-closed $D$ is strict (i.e., contains all
objects of $\cu$ isomorphic to its elements).

The smallest extension-closed $D$ containing a given $D'\subset \obj \cu$ will be called the {\it extension-closure} of $D'$.

We will call the smallest Karoubi-closed triangulated subcategory $\du$ of $\cu$ such that $\obj \du$ contains $D'$ the {\it triangulated subcategory generated by} $D'$.

$\au$ will usually denote some abelian category.
We will 
usually assume that $\au$ satisfies AB5, i.e., 
 is closed with respect to all
small coproducts, and  filtered direct limits of exact sequences in
$\au$ are exact.

We will call a covariant additive functor $\cu\to \au$
for an abelian $\au$ {\it homological} if it converts distinguished
triangles into long exact sequences; homological functors
$\cu^{op}\to \au$ will be called {\it cohomological} when considered
as contravariant functors $\cu\to \au$.

$H:\cu\to \au$ below will always 
be cohomological. 

We will often specify a distinguished triangle by two of its
arrows.
For $f\in\cu (X,Y)$, $X,Y\in\obj\cu$, we will call the third vertex
of (any) distinguished triangle $X\stackrel{f}{\to}Y\to Z$ a cone of
$f$. 

For a class of
objects $C_i\in\obj\cu$, $i\in I$, we will denote by $\lan C_i\ra$
the smallest strict (see above) full triangulated subcategory containing all $C_i$; for
$D\subset \cu$ we will write $\lan D\ra$ instead of $\lan C:\ C\in\obj
D\ra$.  

We will say that $C_i$ {\it  cogenerate} $\cu$ if
$\cu$ is closed with respect to small products and  coincides with its smallest strict triangulated subcategory that fulfills
this property and contains $C_i$ (in  \cite{bger} $C_i$ were called {\it weak cogenerators} of $\cu$. 

 For additive categories $C,D$ we denote by $\adfu(C,D)$ the
category of additive functors from $C$ to $D$
(we will ignore set-theoretic difficulties here since
we will mostly need the categories of functors from those $\cu$ that are skeletally small).

$\ab$ is the category of abelian groups.

For a category $C$ and 
a 
partially ordered 
 index set $I$
 we will call a 
 set $X_i\subset \obj C$ a (filtered)
{\it projective system}
 if for any $i,j\in I$ there exists some maximum, i.e., an $l\in I$
 such that $l\ge i$ and $l\ge j$, and for all $j\ge i$ in $I$ there are fixed morphisms $X_j\to X_i$ satisfying the natural functoriality property. For such a system we have the natural notion of an inverse limit. Dually, we will call the inverse limit of a system of $X_i\in \obj C^{op}$ the direct limit of $X_i$ in $C$.
 
 All limits, colimits, and pro-objects in this paper will be filtered ones.
 
$k$ will be our perfect 
 base field of characteristic $p$; $p$ could be $0$. 

We also list some more definitions and the main notation of this paper.

$t$-structures (along with $\hrt$) and Postnikov towers are considered in \S\ref{dtst}; the categories $\opa\subset \popa$,  pro-schemes, and their 'twists' $X_+\brj$ are defined in \S\ref{sprs} (whereas some additional conventions for them are contained in Remark \ref{rpgysin});  $\sh$ and the $t$-structure $t$ on it, the functor $\om$,   
$\sinf$, the 'twists' $\brj$ and $\brjj$, semi-local pro-schemes, and normal bundles $N_{X,Z}$  are mentioned in \S\ref{ssh}; 
primitive (pro)schemes are defined in \S\ref{sprim},
weight structures (along with $\cu_{w\le i}$, $\cu_{w= i}$, $\cu_{w\ge i}$, $\hw$, $\cu_{[i,j]}$, and weight complexes), negative subcategories and (positive) weight Postnikov towers are defined in \S\ref{swr}; coenvelopes and countable homotopy limits are introduces in \S\ref{snews}; weight filtrations $W^kH^m$, weight spectral sequences $T(H,M)$, and  pure cohomology theories 
are defined in  \S\ref{swfs}; (nice) dualities $\Phi$ of triangulated categories and orthogonal weight and $t$-structures are defined in \S\ref{sort}; the categories $\psh$, $\gdp$, and $\gdb=\ho(\gdp)$ (along  with the extension of $\om$ to $\popa$) are introduced in \S\ref{comot} (and constructed in \S\ref{scgd}); our main (Gersten) weight structure $w$ (for the category $\gd$) is constructed in \S\ref{scwger};
we introduce extended cohomology theories in \S\ref{sextkrau}; Cousin complexes $T_H(-)$ are studied in \S\ref{sext}; cocompact functors $\hw\to \au$ and strictly homotopy invariant Nisnevich sheaves are considered in \S\ref{spure};  $\sht$, $\omt$,  $\shtpl$,  $\shtpl$, $\eta$, and $\tau$ are considered in \S\ref{sht}; (weakly) orientable spectra  are studied in \S\ref{sshinvp} (and related to the Hopf element  $\eta$); $\mgl$, $\mglmod$, $\shmgl=\ho(\mglmod)$, and $\gdmgl$ are considered in \S\ref{smgl}; the category $\gdm$ of comotives is mentioned in \S\ref{sdm}.

\subsection{
$t$-structures, Postnikov towers, and 
idempotent completions}
\label{dtst}

In order to fix the notation, let us recall the definition of a $t$-structure.

\begin{defi}\label{dtstr}

A pair of subclasses  $\cu^{t\ge 0},\cu^{t\le 0}\subset\obj \cu$
for a triangulated category $\cu$ will be said to define a
$t$-structure $t$ if $(\cu^{t\ge 0},\cu^{t\le 0})$  satisfy the
following conditions:

(i) $\cu^{t\ge 0},\cu^{t\le 0}$ are strict (i.e., contain all
objects of $\cu$ isomorphic to their elements).

(ii) $\cu^{t\ge 0}\subset \cu^{t\ge 0}[1]$, $\cu^{t\le
0}[1]\subset \cu^{t\le 0}$.

(iii) {\bf Orthogonality}. $\cu^{t\le 0}[1]\perp
\cu^{t\ge 0}$.

(iv) {\bf $t$-decomposition}. For any $X\in\obj \cu$ there exists
a distinguished triangle
\begin{equation}\label{tdec}
A\to X\to B[-1]{\to} A[1]
\end{equation} such that $A\in \cu^{t\le 0}, B\in \cu^{t\ge 0}$.

\end{defi}

We will need some more notation. 

\begin{defi} \label{dt2}

1. A category $\hrt$ whose objects are $\cu^{t=0}=\cu^{t\ge 0}\cap
\cu^{t\le 0}$, $\hrt(X,Y)=\cu(X,Y)$ for $X,Y\in \cu^{t=0}$,
 will be called the {\it heart} of
$t$. Recall (cf. Theorem 1.3.6 of \cite{bbd}) that $\hrt$ is abelian
(short exact sequences in $\hrt$ come from distinguished triangles in $\cu$).

2. $\cu^{t\ge l}$ (resp. $\cu^{t\le l}$) will denote $\cu^{t\ge
0}[-l]$ (resp. $\cu^{t\le 0}[-l]$).

\end{defi}

\begin{rema}\label{rts}

1. Recall (cf. Lemma IV.4.5 in \cite{gelman}) that (\ref{tdec})
defines additive functors $\cu\to \cu^{t\le 0}:X\to A$ and $C\to
\cu^{t\ge 0}:X\to B$. We will denote $A,B$ by $X^{t\le 0}$ and
$X^{t\ge 1}$, respectively.

The triangle (\ref{tdec}) will be called the {\it t-decomposition} of $X$. If
$X=Y[i]$ for some $Y\in \obj\cu$, $i\in \z$, then we will denote $A$
by $Y^{t\le i}$ (it belongs to $\cu^{t\le 0}$) and $B$ by $Y^{t\ge
i+1}$ (it belongs to  $\cu^{t\ge 0}$), respectively. 
Objects of the type $Y^{t^\le i}[j]$ and
$Y^{t^\ge i}[j]$ (for $i,j\in \z$) will be called {\it
$t$-truncations of $Y$}.

2. We denote by $X^{t=i}$ the $i$-th cohomology of $X$ with respect
to $t$, that is $(X^{t\le i})^{t\ge 0}$ (cf. part 10 of \S IV.4 of
\cite{gelman}). 

3.  The following statements are obvious (and well-known): $\cu^{t\le 0}={}^\perp
(\cu^{t\ge 1})$; $\cu^{t\ge 0}= (\cu^{t\le -1})^\perp$.

4. Our conventions for $t$-structures and 'weights' (see Remark \ref{rstws}(3) below) follow the ones of \cite{bbd}. So, for any $n\in \z$ 'our' $(\cu^{t\le n},\cu^{t\ge n})$ corresponds to $(\cu_{t\ge -n},\cu_{t\le -n})$ in the notation of the papers \cite{morintao}, \cite{tmodel}, and \cite{degorient}.

\end{rema}

Below we will need the notion of  Postnikov tower in
a triangulated category several times (cf. \S IV.2 of \cite{gelman}; Postnikov towers are closely related (in an obvious way) to {\it triangulated exact couples} of \S2.4.1 of \cite{degorient}).

\begin{defi}\label{dpoto}
Let $\cu$ be a triangulated category.

1. Let $l\le m\in \z$.

We will call a {\it bounded Postnikov tower} for $M\in\obj\cu$ the
following data: a sequence of $\cu$-morphisms $(0=)Y_l\to
Y_{l+1}\to\dots \to Y_{m}=M$ along with distinguished triangles
\begin{equation}\label{wdeck3}
 Y_{ i-1} \to Y_{i}\to M_{i}
\end{equation}
for some $M_i\in \obj \cu$;
here $l<i\le m$.

2. An unbounded Postnikov tower for $M$ is a collection of $Y_i$ for
$i\in\z$ that is equipped (for all $i\in\z$) with the following: connecting arrows
$Y_{i-1}\to Y_{i}$ (for $i\in\z$), morphisms $Y_i\to M$ such that all
the corresponding triangles commute, and distinguished triangles
(\ref{wdeck3}).

In both cases, we will denote  $M_{-p}[p]$ by $M^p$; we will call $M^p$ the {\it factors} of our Postnikov tower.

\end{defi}

\begin{rema}\label{rwcomp}
1. Composing (and shifting) arrows from   triangles (\ref{wdeck3})
for two subsequent $i$ one can construct a complex whose terms are
$X^p$ (it is easily seen that this is a complex indeed; 
see Proposition 2.2.2 of \cite{bws}). 

2. Certainly, a bounded Postnikov tower can be easily completed to
an unbounded one. For example, one can take $Y_i=0$ for $i<l$,
$Y_i=M$ for $i>m$; then $M^i=0$ if $i<l$ or $i\ge m$.
\end{rema}

\subsection{On open pairs and pro-schemes}\label{sprs}

$\var\supset \sv\supset \spv$ will denote the class of all varieties
over $k$, resp. of smooth varieties, resp. of smooth projective
varieties. $\sm$ is the category of smooth $k$-varieties.

We consider the category $\opa$ 
of  open 
 embeddings of smooth varieties over $k$ (we will denote the object corresponding to an embedding $U\to X$ as $X/U$).
We have the obvious (componentwise) disjoint union operation for $\opa$.


For $X\in \sv$ we will denote $X\sqcup \pt/\pt$ ($\pt$ is our notation for $\spe k$) by $X_+$. We also define certain (``shifted Tate'') twist on $\opa$ as follows: for $Z=X\setminus U$ we set
$X/U\lan 1\ra=X\times \afo/(X\times \afo\setminus Z\times \ns)$, so that for any $j>0$ we have $X/U\lan j\ra=(X\times \af^j/(X\times \af^j\setminus Z\times \ns))$. This is a functor from $\opa$ into itself. 
Another functor is $\times Y$ for any fixed $Y\in \sv$ that sends $X/U$ into $X\times Y/U\times Y$; we will be interested in the case $Y=\afo$.



We will also need the category $\popa$ of (filtered) 
 pro-objects of $\opa$. 

Obviously, any functor $H:\opa\to \au$ can be naturally extended to a functor from $\popa$  (via direct limits) if $\au$ is an abelian category satisfying AB5.

Below we will treat a special sort of objects of $\popa$. Those may be called 'disjoint unions of intersections of smooth varieties'. Being more precise, we consider inverse limits of objects of the type $X_{i+}$ (for $X_i\in \sv$) such that the restriction of the  structural 
 morphism $X_i'\to X_{i'}$ (for any $i\ge i'$) to any connected component of $X_i$ is either an open embedding or the projection to the 'extra point' of $X_{i'+}$. The corresponding object will be denoted by $X_+$; we will also write $X=\prli X_i$. Note that one can 
 speak of the connected components of $X$ omitting the extra point (we will call them {\it connected pro-schemes}; they are projective limits of smooth connected varieties connected with open dense embeddings); we will write $X=\sqcup_{\al\in A} X^\al$.
 Note that connected pro-schemes (and their finite disjoint unions) often yield actual $k$-schemes.

We also note that any such union yields a pro-scheme; more generally, for any $j\ge 0$, $X^\al=\inli_{i_\al\in I_\al} X_{i_\al}$, we will consider  the object 
$$ X_+\lan j\ra=(\sqcup_{\al\in A} X^\al)_+\lan j\ra=\prli_{B,i_\be\in I_\be \forall \be\in B} (\sqcup_{\be \in B} X_{i_\be})_+\lan j\ra $$ for $B$ running through finite subsets of  $A$. 
Here the transition maps go from $X_{i_\be}\lan j\ra$ either to $X_{i_\be'}\lan j\ra$ 
 for $i_\be' \le_{I_\be} i_\be\in I_\be$ or to the 'extra' $\pt\lan j\ra $
(if the corresponding $B'$ does not contain $\be$).
We also obtain that  $(\sqcup X^\al)_+\brj$  can be presented as the inverse  limit of the pro-objects $\sqcup_{\be\in B}X^\be_+\brj$. 

  One can also speak of  Zariski points of a pro-schemes $X$. 
  We will say that $X$ is of dimension $\le d$ if $X_+\cong \inli X_{i+}$ for some $X_i$ of dimension at most $d$. Alternatively, one can check whether the transcendence degrees of all residue fields of $X$ 
  over $k$ are at most $d$.


\subsection{
$\sh$ and the homotopy $t$-structure on it: reminder}\label{ssh}

We do not give the definition of $\sh$ here; we only recall that it can be defined via a chain of functors
$\dopshs\to \doshs\to \hk\to \sh$; 
 those are homotopy functors for certain left Quillen functors of  closed model categories. Here
 the (underlying category for) the model category corresponding to  $\dopshs$ (resp. to $\doshs$ and $\hk$) is $\dopsh$ (resp. $\dosh$), where
  $\pspt$ (resp. $\spt$) denotes the category of  presheaves (resp. Nisnevich sheaves) of pointed sets on $\sm$.  
We will say more on these model categories and fill in the corresponding details of the proofs in \S\ref{srsmc}
 below.

For $X/U\in\obj \opa$ we define $\om(X/U)$ as the image in $\sh$ of the  
discrete pointed simplicial presheaf $\sm(-,X_+)/\sm(-,U_+)$, 
  i.e., of the presheaf  sending $Y\in \sv$ into $\sm(Y,X)/\sm(Y,U)$ pointed by the 'image' of $\sm(Y,U)$ if the latter is non-empty, and into $\sm(Y,X)\sqcup\pt$ pointed by this point in the opposite case (cf. \S2.3.2 of \cite{degdoc}).
 Recall also that $\sh$ is triangulated monoidal with respect to the operation $\wedge$ compatible with the obvious $\wedge$ for pointed presheaves; 
  it contains the Thom spectrum $T=\om(\afo/\gmm)$ (corresponding to the line bundle $\afo\to \pt$).
We denote the operation $\wedge T^j[-j]$ on $\sh$ by $\{j\}$ (for any $j\ge 0$).

We will need 
 the following properties of $\sh$. 

\begin{pr}\label{psh}

\begin{enumerate}

\item\label{ish1} 
The functor $\sm\to \sh:X\mapsto \om(X_+)$ is exactly the 'usual' one considered in \S4.2 of \cite{morintao}.
 Besides, there is a natural isomorphism $\om(-\brj)\to \om(-)\wedge T^j$ for the spectrum $T$ being the one considered in ibid. 


\item\label{ish6} For any $U\in \sv$ we have $\om (U/U)=0$; for any $X/Y\in \obj \opa$ the functor $\om$ converts the natural morphism
$X/Y\to X\sqcup U/Y\sqcup U$ into an isomorphism.

\item\label{ish3} 
For any $j\ge 0$ and  open embeddings $Z\to Y\to X$ of smooth varieties the 
natural morphisms $\om(Y/Z\brj) \to \om(X/Z\brj) \to \om(X/Y\brj)$
 can be completed to a distinguished triangle.


\item\label{ish4}
For any finite set of $X_i\in \sv$ 
 the obvious morphisms $(\sqcup X_i)_+\brj\to X_{i+}\brj$ yield an isomorphism $\om((\sqcup X_i)_+\brj)\cong \prod \om(X_{i+}\brj) $. 

\item\label{ish7} $\om$ is {\it homotopy invariant}, i.e., 
for any $X\in \sv$ we have $\om(X_+)\cong \om (X\times \afo_+)$

\item\label{ish8} $\om(-_+)$ converts Nisnevich distinguished squares of smooth varieties into distinguished triangles (see Example 4.1.11 of \cite{morintao} for more detail).
 
 
\item\label{ish5} 
Let $j\ge 0$; let $i:Z\to X$ be a closed embedding of smooth varieties,  
and denote by  $B(X,Z)$  the corresponding deformation to the normal cone variety (see the proof of Theorem 3.2.23 of \cite{movo} or \S4.1 of \cite{degdoc}). Then  the natural 
$\opa$-morphisms $X/X\setminus Z\brj \to B(X,Z)/B(X,Z)\setminus Z\times \afo\brj$ and $N_{X,Z}/N_{X,Z}\setminus Z\brj\to B(X,Z)\setminus Z\times \afo\brj$ (where $N_{X,Z}$ is the normal bundle for $i$) become isomorphisms after we apply $\om$.


 
 \end{enumerate}

\end{pr}

\begin{proof}
Assertions \ref{ish1}, \ref{ish6}, \ref{ish7}, and \ref{ish8} are obvious from the well-known properties of $\sh$ (see \cite{morintao}).

\ref{ish3}. By assertion \ref{ish1}, we can assume that $j=0$. In this case it suffices to note that the pair of morphisms in question 
yields a cofibration sequence in $\dopsh$; cf. also \S2.4.1 of \cite{degorient}. 

\ref{ish4}. Again, assume that $j=0$. Obviously, it suffices to prove the statement for the union of two varieties. By assertion \ref{ish6}, in this case it suffices to verify that the distinguished triangle corresponding to $\pt\to \pt\sqcup X_1\to \pt \sqcup X_1\sqcup X_2$ (by the previous assertion) splits. The latter is an easy consequence of assertion \ref{ish6}.


\ref{ish5}. 
Again, we can assume that $j=0$.
In this case it suffices to note that the 
corresponding $\opa$-morphisms  map into isomorphisms already in  $\hk$  (this is exactly Proposition 3.2.24 
of \cite{movo}); all the more 
we obtain isomorphisms in $\sh$.

\end{proof}

The following observation relates twists to shifts (somehow).

\begin{rema}\label{rpsh}
1. In particular, we obtain a distinguished triangle $\om (\gmmpl)\to \om(\afo_+) \to \om (\afo/\gmm)\to \om (\gmmpl)[1]$. 
Similarly, for any $X\in \sv,\ j\ge 0$, we have $\om (X_+\lan j+1\ra )\cong \co (\om (X_+\brj)\to \om(X\times \gmmpl\lan j\ra ))[1] $.  Besides, $\om (X_+\brj)$ is 
obviously a retract of $\om(X\times \gmmpl\lan j\ra ))$; hence  $\co (\om (X_+\brj)\to \om(X\times \gmmpl\lan j\ra ))$ is the 'kernel' of the corresponding projection.

2. Alternatively, for the proof of assertion \ref{ish4} one can note that disjoint union of pro-schemes yields the bouquet operation in $\dopsh$.
\end{rema}

Let us recall the basic properties of Morel's homotopy $t$-structure on $\sh$ (see Theorem 4.3.4
of \cite{morintao}). We will denote it just by $t$; this is the ``main'' $t$-structure of this paper. Recall that in ibid. for any $n\in \z$ the class $\sh^{t\le n}$ was called the one of $1-n$-connected ($\afo$-local) spectra 
(note  that our $\sh^{t\le n}$ is $\sh{}_{t\ge -n}$ in the notation of ibid.; see Remark\ref{rts}(4)). 
For an $E\in \obj \sh$ we will denote by $E^n$ (resp. by $E^n_j$) the application of the functor represented by $E[n]$ to $\opa\subset \popa$ via $\om$ (resp. via $\om \circ \brj$).


\begin{pr}\label{psht}

Let $E\in \obj \sh$. Then the following statements are valid.

1. For any $X\in \sv$ we have $\om(X_+)\in \sh^{t\le 0}$.

2. $E\in  \sh^{t\ge 0}$ if and only if $E^n(X_+)=\ns $ for any $X\in \sv$, $n<0$.

3.  $E\in  \sh^{t\le 0}$ if and only if $E^n(\spe K_{+})=\ns$ for all $n>0$ and for all function fields $K/k$. 
If this is the case, we also have $E_j^{n+j}(\spe K_+)=\ns$ for any such $K,n$, and any $j\ge 0$.

4. 
For a pro-scheme $S$ and some $i\in \z$ assume that $E^0(S_+)=\ns$ for any $E\in \sh^{t\le i}$.
Then we also have $E_j^{n+j}(S_+)=\ns$ for any such $E$ and any $j\ge 0$.

5. If $E\in  \sh^{t\le 0}$ then there exist some $E_i$
belonging to the extension-closure of $\{\coprod \om(S_{i+})[n_i]\}$ for $S_i\in \sv$, $n_i\ge 0$, and a distinguished triangle 
$\coprod E_i\to E\to \coprod E_i[1]$. 


\end{pr}
\begin{proof}

1. Immediate from 
Lemma 4.3.3 of ibid.

2. This is the just the definition of  $\sh^{t\ge 0}$ (see Definition 4.3.1(1) of ibid.).

3. Immediate from Lemmas 4.2.7  and 4.3.11 of ibid.

4. Immediate from 
loc. cit.

5. Immediate from assertion 2 along  with Theorem 12.1 of \cite{kellerw} (here we use the compactness of $\sinf X_{i+}[n_i]$ in $\sh$). 


\end{proof}

Below $\shc$ will denote the triangulated subcategory generated (in the sense described in \S\ref{snot}) by 
 $\{\om (X_+):\ X\in \sv\}$. 
 Note that 
  the objects of this  subcategory are  exactly the compact objects of $\sh$.

Now recall some 
properties of semi-local schemes.

\begin{rema}\label{rsemiloc}
In this paper all semi-local schemes that we consider will be affine essentially smooth ones. 
 Such an $S$ is the semi-localization of a smooth affine variety $V/k$ at a finite collection of Zariski points. 
Obviously, if $f:V'\to V$ is a finite morphism of smooth varieties (in particular, a closed embedding) then $S\times_V V'$ is (affine essentially smooth) semi-local also. 
 It is well-known that all vector bundles over connected semi-local schemes (in our sense) are trivial.

Certainly, any semi-local scheme yields a pro-scheme (with a finite number of connected components). We will call an $\opa$-disjoint union of an arbitrary set of (affine smooth) semi-local schemes   a {\it semi-local pro-scheme}.
 \end{rema}


We will need the following important 
result (that is essentially well-known also).

\begin{pr}\label{pshinv}
Assume that $k$ is infinite; let $X$ be a 
semi-local pro-scheme, $E\in \sh^{t\le 0}$. 
Then for any $j\ge 0, i>j$, we have $E^{i}_j(X_+)=\ns$ (see the notation of the previous subsection).

\end{pr}
\begin{proof}

Obviously, we can assume that $X$ is connected.

By Proposition \ref{psh}(\ref{ish3}),
 $X\mapsto E_j^*(X_+)$ 
 along  with $(X,U)\mapsto E_j^*(X/U)$  yields a {\it cohomology theory with supports} in the sense of Definition 5.1.1(a) of \cite{suger}. The Nisnevich excision for 
 $\om(-_+)$
  yields axiom COH1 of loc. cit., whereas Proposition \ref{psh}(\ref{ish7}) 
is precisely 
 axiom COH3 of (\S5.3 of) ibid. Hence we can apply Theorem 6.2.1 of ibid.; we obtain an injection $E^{i+j}_j(X_+\brj)\to E^{i+j}_j(X_{0+}\brj)$. In order to conclude the proof it suffices to apply Proposition \ref{psht}(3,4).


\end{proof}

\subsection{Primitive (pro)schemes: definition}\label{sprim}

in \S\ref{ssupl} below we will consider the stable motivic homotopy category $\sht$ and its '$\tau$-positive part' $\shtpl$. It turns out that in the latter category one can formulate a connectivity result for a class of (pro)schemes   wider than that of semi-local ones. So, we will need the following definition.

\begin{defi}\label{dprim}

If $k$ is infinite then a pro-scheme will be called primitive if all
of its connected components are affine (essentially smooth) $k$-schemes and their coordinate rings
$R_j$ satisfy the following primitivity criterion: for any $n>0$ 
every polynomial in $R_j[X_1,\dots,X_n]$ whose coefficients generate
$R_j$ as an ideal over itself, represents an $R_j$-unit.

If $k$ is finite, 
then we will  (by an abuse of notation, in this paper) call a pro-scheme primitive
whenever it is semi-local in the sense of  Remark \ref{rsemiloc} (i.e., if all of its connected components are 
affine essentially smooth semi-local). 

\end{defi}

\begin{rema}\label{rabprim}

Recall that in the case of infinite $k$ all semi-local $k$-algebras
satisfy the primitivity criterion (see Example 2.1 of
\cite{walker}).
\end{rema}

We will need the following properties of primitive pro-schemes.

\begin{pr}\label{pprim}

1. Assume that a primitive pro-scheme $S$ is the projective limit of some open subvarieties of a $V\in \sv$. Let 
$f:V'\to V$ be a finite morphism of smooth varieties. Then $S\times_V V'$ is primitive also.

2. Any finitely generated module of constant rank over a primitive 
ring is free.
\end{pr}
\begin{proof}
Essentially, we have already noted (in Remark \ref{rsemiloc}) that these statements are valid if $k$ is finite (since in this case primitive pro-schemes are just semi-local ones by our convention). In the case where $k$ is infinite (so that our notion of primitivity coincides with the one used by Walker) these assertions are given by 
 Theorems 4.6 and 2.4 of \cite{walker}, respectively.
\end{proof}

\section{Weight structures: reminder  and 
the case of cocompactly cogenerated categories}\label{sws}

In this section we recall the formalism of weight structures and their relation to ({\it orthogonal}) $t$-structures. The only 
 new results of this section are Theorem \ref{tnews}(II,III) 
 and Corollary \ref{cwss}. The author suspects that the dual to Theorem \ref{tnews} (on weight structures in  compactly 
 generated triangulated categories) can  have interesting applications beyond the scope of the current paper.


\subsection{
Basic definitions and properties}\label{swr}

\begin{defi}\label{dwstr}

I For a triangulated category $\cu$, a pair of classes $\cu_{w\le 0},\cu_{w\ge 0}\subset\obj \cu$ 
will be said to define a weight
structure $w$ for $\cu$ if 
they  satisfy the following conditions:

(i) $\cu_{w\ge 0},\cu_{w\le 0}$ are 
 Karoubi-closed in $\cu$
(i.e., contain all $\cu$-retracts of their objects).

(ii) {\bf Semi-invariance with respect to translations.}

$\cu_{w\le 0}\subset \cu_{w\le 0}[1]$, $\cu_{w\ge 0}[1]\subset
\cu_{w\ge 0}$.

(iii) {\bf Orthogonality.}

$\cu_{w\le 0}\perp \cu_{w\ge 0}[1]$.

(iv) {\bf Weight decompositions}.

 For any $M\in\obj \cu$ there
exists a distinguished triangle
\begin{equation}\label{wd}
B\to M\to A\stackrel{f}{\to} B[1]
\end{equation} 
such that $A\in \cu_{w\ge 0}[1],\  B\in \cu_{w\le 0}$.

II The full category $\hw\subset \cu$ whose object class is
$\cu_{w=0}=\cu_{w\ge 0}\cap \cu_{w\le 0}$ 
 will be called the {\it heart} of 
$w$.


III $\cu_{w\ge i}$ (resp. $\cu_{w\le i}$, resp.
$\cu_{w= i}$) will denote $\cu_{w\ge
0}[i]$ (resp. $\cu_{w\le 0}[i]$, resp. $\cu_{w= 0}[i]$).

IV We denote $\cu_{w\ge i}\cap \cu_{w\le j}$
by $\cu_{[i,j]}$ (so it equals $\ns$ for $i>j$).

V We will  
call
$\cu^b=(\cup_{i\in \z} \cu_{w\le i})\cap (\cup_{i\in \z} \cu_{w\ge i})$ the class of {\it bounded} 
objects of $\cu$. We will say that $w$ is bounded if $\cu^b=\obj \cu$.

Besides, we will call $\cup_{i\in \z} \cu_{w\le i}$ the class of {\it bounded above} 
objects. 

VI $w$ will be called 
{\it non-degenerate from below (resp. from above)} if $\cap_{l\in z} \cu_{w\le l}=\ns$ (resp. $\cap_{l\in \z} \cu_{w\ge l}=\ns$).

VII Let $C$ be a 
full additive subcategory of a triangulated $\cu$.

We will say that $C$ is {\it negative} if
 $\obj C\perp (\cup_{i>0}\obj (C[i]))$.

VIII We will call a Postnikov tower for $M$
(see Definition \ref{dpoto})
 a {\it weight Postnikov tower} if
all $Y_j$ are some choices for $w_{\le j}M$.
In this case we will call the corresponding complex whose terms are $M^p$
(see Remark \ref{rwcomp})
a {\it weight complex} for $M$.

We will call a weight Postnikov tower for $M$ {\it positive} if 
we take $w_{\le r}M=0$ for all $r<0$ (and so, $M\in \cu_{w\ge 0}$).

\end{defi}

\begin{rema}\label{rstws}

1. A weight decomposition (of any $M\in \obj\cu$) is (almost) never canonical;
still we will sometimes denote (any choice of) a pair $(B,A)$ coming from in (\ref{wd}) by $(w_{\le 0}M,w_{\ge 1}M)$. 
For an $l\in \z$ we denote by $w_{\le l}M$ (resp. $w_{\ge l}M$) a choice of  $w_{\le 0}(M[-l])[l]$ (resp. of $w_{\ge 1}(M[1-l])[l-1]$).


2. A  simple (and yet  useful) example of a weight structure comes from the stupid
filtration on the homotopy categories of cohomological complexes
$K(B)$ for an arbitrary additive  $B$. 
In this case
$K(B)_{w\le 0}$ (resp. $K(B)_{w\ge 0}$) will be the class of complexes that are
homotopy equivalent to complexes
 concentrated in degrees $\ge 0$ (resp. $\le 0$).  The heart of this weight structure 
is the Karoubi-closure  of $B$ in $K(B)$.
 in the corresponding category.  

3. In the current paper we use the 'homological convention' for weight structures; 
it was previously used in \cite{hebpo}, \cite{wildic},  and 
\cite{brelmot},
  whereas in 
\cite{bws} and in \cite{bger} the 'cohomological convention' was used. In the latter convention 
the roles of $\cu_{w\le 0}$ and $\cu_{w\ge 0}$ are interchanged, i.e., one considers   $\cu^{w\le 0}=\cu_{w\ge 0}$ and $\cu^{w\ge 0}=\cu_{w\le 0}$. So,  a complex $M\in \obj K(B)$ whose only non-zero term is the fifth one 
 has weight $-5$ in the homological convention, and has weight $5$ in the cohomological convention. Thus the conventions differ by 'signs of weights'; 
 $K(B)_{[i,j]}$ is the class of retracts of complexes concentrated in degrees $[-j,-i]$. 
  
 
4. In \cite{bws} the axioms of a weight structure also 
required $\cu_{w\le 0}$ and $\cu_{w\ge 0}$ to be additive. Yet it is not necessary to include this property into the axioms since it follows from the remaining ones; see Proposition 1.3.3(1,2) of ibid.

5. Also, in the current paper we shift the numeration for $Y_i$ (in the definition of weight Postnikov tower) by $[1]$ if compared with \cite{bger}.  
 
\end{rema}

Now  we recall some basic 
properties of weight structures. 

\begin{pr} \label{pbw}
Let $\cu$ be a triangulated category endowed with a weight structure $w$, $M\in \obj \cu$, $i\in \z$. Then the following statements are valid.

\begin{enumerate}

\item \label{idual}
The axiomatics of weight structures is self-dual, i.e., for $\du=\cu^{op}$
(so $\obj\cu=\obj\du$) there exists the (opposite)  weight
structure $w'$ for which $\du_{w'\le 0}=\cu_{w\ge 0}$ and
$\du_{w'\ge 0}=\cu_{w\le 0}$.

\item\label{iextw} 
 $\cu_{w\le i}$, $\cu_{w\ge i}$, and $\cu_{w=i}$
are Karoubi-closed and extension-closed in $\cu$ (and so, additive). 

Besides, if $M\in \cu_{w\le 0}$, then $w_{\ge 0}M\in \cu_{w=0}$ (for any choice of $w_{\ge 0}M$).

\item \label{isump}
If (in $\cu$) we have a distinguished triangle $A\to B\to C$ for $B\in \cu_{w=0}$, $C\in \cu_{w\ge 1}$, then $A\cong B\bigoplus C[-1]$.

    \item\label{iwpost} For any choice of $w_{\le j}M$ for $j\in \z$
    there exists  a weight Postnikov tower for $M$.
Besides, for any weight Postnikov
tower we have
     $\co(Y_i\to M)\in \cu_{w\ge i+1}$;
    $M^i\in \cu_{w=0}$.

\item\label{iwpostc}
    Conversely, any bounded Postnikov tower (for $M$)
with $M^j\in \cu^{w=0}$ for all $j\in \z$ is a weight Postnikov
    tower for it.

\item\label{iwefun}
Weight decompositions are {\it weakly functorial}, i.e., any $\cu$-morphism of objects has a (non-unique) extension to a morphism of (any choices of) their weight decomposition triangles.

 \item\label{iwprod}   $\cu_{w\ge i}$ is closed with respect to
all those small
 products that  exist in $\cu$. 

 \item\label{iort}
 $\cu_{w\ge i}=(\cu_{w\le i-1})^{\perp}$.
 
 \item\label{icub} $w$ induces
a bounded  weight structure for $\cu^b$ (i.e., we consider the corresponding full subcategory of $\cu$ and the classes $\cu_{w\le 0}\cap \cu^b$ and $\cu_{w\ge 0}\cap \cu^b$ in it), whose
heart equals $\hw$. 

\item\label{igenlm}

For any $l\le m \in \z$ the class $\cu_{[l,m]}$ is the smallest Karoubi-closed extension-stable subclass of $\obj\cu$
containing $\cup_{l\le j\le m}\cu_{w=j}$.

\item\label{impostp} 
Assume that 
 $M,M'\in \cu_{w\ge 0}$; suppose that
 $f\in\cu(M,M')$ can be extended to a morphism of
 (some of) their positive Postnikov towers  that establishes an isomorphism
$M^0\to M'{}^0$. Assume also that $M'\in \cu_{w=0}$.
Then $f$ yields a projection of $M$ onto $M'$ (i.e., $M'$ is a retract of $M$ via $f$).

\item\label{iwc} $M$ determines its weight complex $t(M)$ up to homotopy equivalence (i.e., up to a $K(\hw)$-isomorphism).
In particular, if $M\in \cu_{w=0}$, then any choice of 
$t(M)$ is $K(\hw)$-isomorphic to $\dots\to 0\to M\to 0\to\dots$.

\item\label{iwcex} 
If $M_0\stackrel{f}{\to} M_1\to M_2$ is a distinguished triangle in $\cu$, then $f$ induces some morphism $t(M_0)\to t(M_1)$ (for any possible choices of these complexes) such that a cone of this morphism yields a weight complex for $M_2$.  

\item\label{iwcoh} Let $H_0:\hw^{op}\to\au$ ($\au$ is an arbitrary abelian category) be an additive functor. Choose a weight complex $t(M)=(M^j)$ for each $M\in \obj \cu$, and denote by 
$H(M)$ the zeroth cohomology of the complex $H_0(M^{-j})$. Then $H(-)$ yields a cohomological functor 
that does not depend on the choices of weight complexes.  

\item\label{isplwc} For $M\in \cu_{w=0}$ consider some (other) positive 
weight Postnikov tower for it; let $t(M)=(M^j)$ denote the corresponding weight complex. Then there exists some  $N^j\in \cu_{w=0},\ j\le 0$, such that 
$t(M)$  is $C(\hw)$-isomorphic to $M^0\bigoplus (\bigoplus_{j\le 0}(N^j\stackrel{\id_{N^j}}{\to}N^j)[-j])$ (i.e., $N^j$ is put in degrees $j-1$ and $j$). 


\item\label{ipostn} For $M,M'\in\obj\cu$  suppose 
 $f\in \cu(M,M')$ is compatible with an isomorphism $w_{\le i}M\to
w_{\le i}M'$. 
Then $\co f\in \cu_{w\ge i+1}$.

\item\label{ikar1} If $\cu$ is Karoubian, then  $\cu_{w\le i}$, $\cu_{w\ge i}$, $\cu_{w=i}$ also are. In particular, this is the case if $\cu$ is closed with respect to all countable products.

\item\label{ikar2} Assume that $\cu$ and $\cu_{w\le 0}$ are closed with respect to arbitrary small products. Then 
the same is true for all $\cu_{w=j}$,  small products of weight decompositions are weight decompositions, and small products of weight Postnikov towers (resp. of weight complexes for some $M_l\in \obj \cu,\ l\in L$) are weight Postnikov towers (resp. is a weight complex for $\prod_l M_l$).

\item\label{icharge} If $M\in \cu^b$ then $M\in \cu^{\ge i}$ if and only if $\cu^{w=j}\perp M$ for all $j<i$. 

\end{enumerate}
\end{pr}

\begin{proof} 

Assertions \ref{idual}--\ref{impostp} 
 are essentially contained in Theorem 2.2.1 of \cite{bger} (whereas their proofs relied on  \cite{bws}; pay attention to Remark \ref{rstws}(3)!). 

\ref{iwc}. The homotopy equivalence of all possible weight complexes is given by Theorem 3.2.2(II) of \cite{bws} (along  with Proposition 3.1.8 of ibid.). It remains to note that for an $M\in \cu_{w=0}$ one obtains $\dots\to 0\to M\to 0\to\dots$ ($M$ is on the $0$-th position) as a weight complex by setting $(w_{\le j}M=) Y_j=M$ for $j\ge 0$ and $Y_j=0$ otherwise.

\ref{iwcex}. Immediate from Theorem 3.3.1(I) of ibid.

\ref{iwcoh}. Immediate from two previous assertions
 (see also Remark 3.1.7(2) of ibid.).


\ref{isplwc}. 
Assertion \ref{iwc} yields that that $t(M)$ is $K(\hw)$-isomorphic to the complex $\dots\to 0\to M\to 0\to\dots$. Hence the corresponding complex $M'=\dots M^{-2}\to M^{-1}\to M^{0}\to M\to 0\to\dots$ is zero in $K(\hw)$. Thus  
it suffices to verify by induction the following fact: if a bounded above complex $N=\dots\to N^{j-2}\to N^{j-1} \stackrel{d^{i-1}}\to N^{j}\to 0\to\dots$ is zero in $K(\hw)$, then $d^{j-1}$ is a projection of  
$N^{j-1}$ onto a direct summand in $\hw$. Now, 
  $d^{j-1}$ is split by the corresponding component of (any) contracting homotopy for $N$. Hence $\co d^{j-1}[-1]$ is the complement of $N^j$ to $N^{j+1}$ in $\cu$. Since $\hw$ is Karoubi-closed in $\cu$, we obtain the result.

\ref{ipostn}. Note first that the octahedral axiom of triangulated categories yields that $\co f\cong \co(w_{\ge i+1}M\to
w_{\ge i+1}M')$. Hence the result follows from the extension-closedness of $\cu_{w\ge i+1}$ (see assertion \ref{iextw}).

\ref{ikar1}. Certainly, a Karoubi-closed class of objects in a Karoubian category is Karoubian. It remains to note that $\cu$ is Karoubian if it is closed with respect to countable products; this is 
Remark 1.6.9 of \cite{neebook}.

\ref{ikar2}. The first part of the assertion immediately follows from assertion \ref{iwprod}. In order to establish the second one it suffices to note that products of distinguished triangles are distinguished; see Remark 1.2.2 of \cite{neebook}.
The third part of the assertion follows immediately.

\ref{icharge}. Certainly, for any $M\in \cu_{\ge i}$ we have $\cu_{w=j}\perp M$ for all $j<i$ (by the orthogonality axiom of weight structure).
So, we should prove the converse application for a bounded $M$. By assertion \ref{icub} we can assume that $w$ is bounded. 
Next, assertion \ref{iort} yields that it suffices to prove that $\cu_{[l,i-1]}\perp M$ for any $l<i$. Hence assertion \ref{igenlm} yields the result.

\end{proof}

\begin{rema}\label{rwgen} 

1. Certainly, in assertion \ref{isplwc}  the corresponding $N^i$ is an $\hw$-direct summand both of $M^i$ and of $M^{i-1}$ (for any $i\in \z$). 

2. A very useful statement (that was 
 applied in several papers)
is that any negative 
generating 
subcategory  $N\subset \cu$ (here we assume that $\cu$ does not contain any proper strict triangulated subcategories that contain $N$) yields a weight structure $w$ for $\cu$  such that $N\subset \hw$. Yet this is rather a tool for constructing bounded weight structures for 'small' triangulated categories; we will have to prove an alternative existence statement below. 

3. Actually, one does not have to pass to the Karoubi-closure in assertion \ref{icub}. Besides, assertion \ref{icharge} is valid for any $M\in \cup_{j\in z}\cu_{w\ge j}$. In the proof we could have considered weight Postnikov towers (for objects of $(\cup_{j\in z}\cu_{w\ge j})\cap \cu_{w\le i-1}$).
\end{rema}

\subsection{Weight structures on cocompactly cogenerated 
categories}\label{snews}


This subsection is mostly dedicated to (the proof of) 
Theorem \ref{tnews}. 
 Part I of it was essentially proved in \cite{paukcomp} (in the dual form); the dual to this statement 
can also be found in \cite{brmz}. Yet other parts of 
the Theorem are new. 

Our construction of weight structures requires countable 'triangulated' homotopy limits (in general triangulated categories). The latter are dual to countable homotopy colimits  introduced in \cite{bockne} (so in order to prove basic properties of the construction we dualize the corresponding results of ibid. and of \cite{neebook}). We will only apply the results of this subsection to triangulated categories closed
with respect to arbitrary small products; so 
the homotopy limits will always exist.

\begin{defi}\label{dcoulim}

For a sequence of objects $Y_i$ (starting from some
$j\in\z$)  and maps $\phi_i:Y_{i}\to Y_{i-1}$, $D=\prod Y_i$,  
we consider the morphism $d:\prod
\id_{Y_i}+ \prod (-\phi_i): D\to D$. 
 Denote  a cone of $d[-1]$ by $Y$. We will write $Y=\prli Y_i$ and
 call $Y$ the {\it homotopy limit} 
  of $Y_i$. 

\end{defi}

\begin{rema}\label{rcoulim}

1. Note that these homotopy limits are not really canonical and functorial in $Y_i$ since the choice of a cone is not canonical, i.e., limits are only defined up to non-canonical isomorphisms.

2. By Lemma 1.7.1 of \cite{neebook}, the homotopy limit of
$Y_{i_j}$ is the same for any subsequence of $Y_i$. In particular,
we can discard any (finite) number of first terms in $Y_i$.

3. By Lemma 1.6.6 of \cite{neebook}, the homotopy limit of
$M\stackrel{\id_M}{\leftarrow}M\stackrel{\id_M}{\leftarrow}
M\stackrel{\id_M}{\leftarrow} M\stackrel{\id_M}{\leftarrow}\dots$ is $M$. Hence
if  $\phi_i$ are
isomorphisms for $i\gg 0$ and $M_i\cong M$ then  $\prli M_i\cong M$.

\end{rema}

We also recall the behaviour of homotopy limits  under (co)representable
functors.

\begin{lem}\label{coulim} 
1.  For any $C\in\obj\cu$ we have a natural surjection $
\cu(C,Y)\to \prli \cu(C,Y_i)$.

2. This map is bijective if all $\phi_i[-1]_*:
\cu(C,Y_{i+1}[-1])\to \cu(C,Y_i[-1])$  are surjective for 
$i\gg
0$.

3. If $C$ is cocompact then $\cu(Y,C)= \inli \cu(Y_i,C)$.
\end{lem}

Below we will also need the following new (simple) piece of homological algebra: the  definition of a coenvelope and some of its properties. 


\begin{defi}\label{dcoe}

For a class $C'\subset \obj\cu$
its {\it coenvelope} is the smallest subclass of $\obj\cu$ that contains $C'\cup\ns$, is closed with respect to arbitrary (small) products,
 and satisfies the following property: 
for any $\phi_i:Y_{i}\to Y_{i-1}$ such that $Y_0\in C$, $\co(\phi_i)[-1]\in C$ for all $i\ge 1$, we have $\prli_{i\ge 0} Y_i\in C$ (i.e., $C$ contains all possible cones of the corresponding distinguished triangle; note that those are isomorphic).
\end{defi}

Now let us prove 
a simple property of this notion.

\begin{lem}\label{lbes}

Suppose that for $C',D\subset \obj \cu$, 
we have 
$ D\perp C'$. Then for the 
 coenvelope $C$ of $C'$ we also have $D\perp C$.

\end{lem}
\begin{proof}

 Since for any $d\in D$ the functor $\cu(d,-)$ converts arbitrary products into products,
it suffices to verify  (for any $d\in D$) the following statement: if for $Y_i$ as in
 Definition \ref{dcoe} 
 we have $d\perp Y_0 $, $d\perp \co(\phi_i)[-1]$ for all $i\ge 1$, then $p\perp \prli Y_i$.
 Now, for any $i\ge 1$ we have a long exact sequence
 $$\dots \to \cu(d, Y_{i}[-1]) \to \cu(d,Y_{i-1}[-1])  \to \cu(d,\co(\phi_i)[-1]) (=\ns) \to \cu(d, Y_{i}) \to \cu(d,Y_{i-1})\to \dots  $$
 Hence $\cu(d,Y_{i}[-1])$ surjects onto $  \cu(d,Y_{i-1}[-1])$, whereas the obvious induction yields that $\cu(d,Y_j)=\ns$ for any $j\ge 0$. Thus Lemma \ref{coulim}(2) yields that  $\cu(d,\prli Y_i)\cong \prli \cu(d,Y_i)=\ns$. 

\end{proof}

\begin{theo}\label{tnews}

Let $\cu$ be triangulated category that is closed with respect to all small products;
let $C'\subset \cu$ be a 
 set of cocompact objects such that $C'[1]\subset C'$. 
Then for the classes $C_1={}^{\perp}(C'[1])$ and 
$C_2$ being the coenvelope of $C'$ the following statements are valid.

I They yield a weight structure on $\cu$, i.e., there exists a $w$ such that $\cu_{w\le 0}=C_1$, $\cu_{w\ge 0}=C_2$.


II
 Denote by $\cu'$ the full subcategory of $\cu$ whose class of objects 
  is the coenvelope of $\cup_{i\le 0}C'[i]$, $\cuperp$ is  
the subcategory whose objects are ${}^\perp \lan C'\ra$. 

1. $\cupr$ and $\cuperp$ are triangulated and closed with respect to all small products.

2. $C_1'=C_1\cap \obj \cupr$ and $C_2\subset \obj \cupr$ yield a weight structure $w_{\cupr}$ on $\cupr$.

3. 
 $w_{\cupr}$ is non-degenerate from below. 

 4. $C_1$ is the extension-closure of $C_1'\cup \cuperp$ in $\cu$.

5. $C_1'$ is (also) closed with respect to all products. 

6. The heart of 
$w_{\cu'}$ is closed with respect to all small products; products of weight decompositions in $\cupr$ are weight decompositions;  small products of weight Postnikov towers (resp. of weight complexes $t(M_i)$ for $M_i\in \obj \cupr$) are weight Postnikov towers (resp. 
yield $t(\prod M_i)$).

7. $\cu$, $\cu'$, $C_1$, $C_1'$, and $C_2$ are Karoubian.
 
III For each $c\in C'$ 
fix a weight complex $(c^i)$ (with respect to $w_{\cupr}$). Then $\hw_{\cu'}$ is equivalent to the Karoubization of the category of all (small) products of $c^i$ for $c$ running through  all objects $C'$, $i\in \z$.

\end{theo}
\begin{proof}

I Obviously, $(C_1,C_2)$ are Karoubi-closed in $\cu$, $C_1\subset C_1[1]$, $C_2[1]\subset C_2$.
Besides, $C_1\perp C_2[1]$ by Lemma \ref{lbes}. 

It remains to verify that any $M\in \obj \cu$ possesses a weight decomposition with respect to $(C_1,C_2)$. We apply (a certain modification of) the method used in the proof of Theorem 4.5.2(I) of \cite{bws} (cf. also the construction of  crude cellular towers in \S I.3.2 of \cite{marg}). 

  We construct a certain sequence of $M_k$ for $k\ge 0$ by induction on $k$ starting
from $M_0=M$. Assume that $M_k$ (for some $k\ge 0$) is already constructed. Then we take $P_k=\prod_{(c,f):\,c\in C',f\in \cu(c,M_k)}c$ and  $M_{k+1}[1]$ being a cone of the morphism $\prod_{(c,f):\,c\in C',f\in \cu(c,M_k)}f:M_k\to P_k$.

Let us 'assemble' $P_k$. 
The compositions of the morphisms $h_k:M_{k+1}\to M_{k}$ given by this construction yields morphisms $g_i:M_i\to M$ for all $i\ge 0$. Besides, the octahedral axiom of triangulated categories 
immediately yields that 
$\co (h_k)\cong P_k$. Now, we complete $g_k$ to distinguished triangles $M_k\stackrel{g_k}{\to}M \stackrel{a_k}{\to}A_k$.
 The octahedral axiom yields the existence of morphisms $s_i:A_{i+1}\to A_{i}$ that are compatible with $a_i$, such that $\co (a_i)\cong P_{i+1}[1]$ for all $i\ge 0$.

We consider $A=\prli A_k$; by Lemma \ref{coulim}(1) $(a_k)$ lift to a certain morphism $a: M\to A$.
 We complete $a$ to a distinguished triangle $B\stackrel{b}{\to} M\stackrel{g}{\to} A\stackrel{f}{\to} B[1]$.
This triangle will be our candidate for a weight decomposition of $M$.

First we note that $A_0=0$; since $\co (a_i)\cong P_{i+1}[1]$, we have
 $A\in C_2$ by the definition of the latter.

It remains to prove that $B\in C_1[-1]$, i.e.,  $B\perp C'$. 
For a $c\in C'$ we should check that
$\cu(B,c)=\ns$. The long exact sequence  $$\dots  \to \cu(A,c)\to  \cu(M,c)\to \cu(B,c)\to \cu(A[-1], c)\to \cu(M[-1],c)\to\dots $$
translates this into the following: $\cu(c,-)(a)$ is surjective and $\cu(-,c)(a[-1])$ is injective. 
Now, by Lemma \ref{coulim}(3), $\cu(A,c)\cong\inli \cu(A_i,c)$ and $\cu(A[-1],c)\cong\inli \cu(A_i[-1],c)$. Hence the long exact sequences
$$\dots  \to \cu(A_k,c)\to  \cu(M,c)\to \cu(M_k,c)\to \cu(A_k[-1], c)\to \cu(M[-1],c)\to\dots $$
yield that it suffices to verify that $\inli\cu(M_k,c)=\ns$ (note here that $h_k$ are compatible with $s_k$). 
Lastly, 
 $\cu(P_k,c)$ surjects onto $\cu(M_k,c)$; hence 
the group $\cu(M_k,c)$ dies in $\cu(M_{k+1},c)$ for any $k\ge 0$ and we obtain the result.

II 1. Obvious. Note here that the cocompactness of elements of $C'$ yields that $\cuperp$ is closed with respect to products; this also yields assertion II.5.

2. Certainly, in order to verify the existence of $w_{\cu'}$ it suffices to verify that the corresponding weight decompositions exist in $\cu'$. We check it for an $M\in \obj \cupr$ 
via applying the duals to several results of
\cite{neebook}.

Theorem 8.3.3 of 
ibid. yields that 
$\cupr$ satisfies the dual to the Brown representability condition (see Definition 8.2.1 of ibid.). Indeed, $\cupr$ is 
cogenerated by $C'$ in the sense described in \S\ref{snot}, 
 whereas the latter notion is dual to Definition 3.2.9 of ibid..
Hence (the dual to) Theorem 8.4.4 of ibid. yields the existence of  
 an exact functor $F:\cu\to \cupr$ that is left adjoint to the embedding 
of $\cupr$ into $\cu$ (the exactness of $F$ is follows from Lemma 5.3.6 of ibid.). 
Certainly, $F$ is isomorphic to the Verdier localization of $\cu$ by $\cuperp$. 
 Since $\cuperp$ is a colocalizing subcategory in $\cu$ (i.e., 
it is closed with respect to small products), 
  $F$ also respects all products (see Corollary 3.2.11 of ibid).  

Now  apply $F$ to (\ref{wd}). We certainly have $F(M)\cong M$, $F(A)\cong A$. The adjunction for $F$ also yields that $F(B)\in C_1$; hence $F(B)\in C_1'$ and we obtain a weight decomposition for $M$ in $\cupr$.

3. In order to verify that $w_{\cu'}$ is non-degenerate from below we should verify for a non-zero  $M\in \obj \cupr$ 
that there exist   
$i\in \z$ and $c\in C$ such that $\cu(M,c[i])\neq 0$. 
The latter is immediate from the fact that 
 the full subcategory of $\cu$ whose objects are
$\{M[-i],\ i\in\z\}^{\perp}$  is triangulated, strict, and closed with respect to all small products.


4. 
As we have already noted, $F$ projects $C_1$ onto $C_1'$. Besides, the adjunction transformation yields for any $N\in C_1$ a distinguished triangle $N'\to N\to F(N)$ for some $N'\in \obj \cuperp$. 
On the other hand, $C_1$ certainly contains $C_1'$ and $\cuperp$, and is extension-closed.


6. Immediate from the previous assertion by Proposition \ref{pbw}(\ref{ikar2}).

7. Immediate from Proposition \ref{pbw}(\ref{ikar1}).


III As we have just shown, $\hw_{\cupr}$ is Karoubian and closed with respect to all products. Since it also contains all $c^i$, it contains the Karoubization in question. 
  
  Now let $M\in \cupr_{w_{\cupr}=0}$. Lemma \ref{lrcomp}  (along  with Proposition \ref{pbw}(\ref{isplwc})) yields that 
  it suffices to verify the existence of a choice for  $t(M)=(M^i) 
  $   such that $M^0$ belongs to the category of products of all $c^i$.
  
  Now consider the full subcategory of $\cupr$ whose objects possess weight complexes all of whose terms are products of certain $c^i$. Obviously, this category contains $C'$. 
   Hence Proposition \ref{pbw}(\ref{iwcex}) yields that it is triangulated. Moreover, assertion II6 yields that it is closed with respect to all small products. Hence this subcategory coincides with all $\cu'$, and so contains $M$. This concludes the proof.

\end{proof}

Below we will need two 
 simple methods of producing a new category with a weight structure out of an old one.

\begin{rema}\label{rrcoeff}

1. Let 
$S\subset \z$ be a set of prime numbers. Then for any triangulated $\cu$ that is  cocompactly cogenerated  by a set of objects one can consider the (following version) of the category $\cu[S\ob]$. Take the colocalizing subcategory (cf. the proof of part II.2 of the theorem) $\cu_{S-tors}$ cogenerated by cones of $c\stackrel{\times s}{\to} c$ for $c\in \obj \cu,\ s\in S$ (i.e., we consider the smallest triangulated subcategory of $\cu$ that contains these objects and is closed with respect to small products; it is easily seen 
to be cocompactly cogenerated by $\co(c\stackrel{\times s}{\to} c)$ for $c$ being compact); then set $\cu[S\ob]=\cu/\cu_{S-tors}$ (the Verdier localization). 

The dual to this construction was described in detail in Appendix A.2 of \cite{kellyth} (and also in Appendix B of \cite{levconv}). So, the results of ibid. yield that the localization functor $l:\cu\to \cu[S\ob]:\ M\mapsto M[S\ob]$ exists (i.e., the morphism groups in the target are sets); $l$ respects (small) products and cocompact objects (and so, cocompact cogenerators); for a cocompact object $c$ of $\cu$ and any $c'\in \obj \cu$ we have $\cu[S\ob](c'[S\ob],c[S\ob])\cong \cu(c',c)\otimes_\z \z[S\ob]$.

Hence all the statements of Theorem  \ref{tnews} will remain true if replace $\cu$ by $\cu[S\ob]$. Indeed, note that 
all our statements and constructions "commute with $l$". 


2. 
Now assume that the category $C^0=\lan l(C')\ra_{\cu[S\ob]}$ decomposes into the direct sum of two (triangulated) subcategories $C'^1,C'^2$; denote by $pr'$ the projection of $C^0$ onto 
$C'^1$ (this is an exact functor, and $pr'(M)$ is a retract of $M$ for any $M\in \obj C^0$).  Then the same 
 argument as the one used in the proof of part II.2 of ibid. yields the following: $pr'$ extends uniquely to a projection $pr$ of $\cu[S\ob]$ onto 
  its subcategory $C^1$ that is cogenerated by $C'^1$ (i.e., to the smallest triangulated subcategory of $\cu[S\ob]$ that contains $C'^1$ and is closed with respect to all products). Moreover, since $pr$ yields a retraction of each object of $\cu[S\ob]$, it is $w_{\cu[S\ob]}$-exact; hence $w_{\cu[S\ob]}$ restricts to a weight structure on $C^1$.
 

Lastly, note  that both of these constructions 
do not require any 'models'. 

3. Another interesting statement that can be proved in 
the general context of Theorem \ref{tnews} is the description of all {\it pure extended} cohomology theories for $\cupr$; cf. \S\ref{spure} below.

4. The author suspects that dualizing Theorem \ref{tnews} would 
be useful for its application in settings not related to Gersten weight structures.
Indeed, $\cu^{op}$ is closed with respect to arbitrary small coproducts, whereas $\cupr^{op}$ is compactly generated (and it seems that compactly generated triangulated categories are "more popular'' than cocompactly cogenerated ones).

\end{rema}

\subsection{Weight filtrations and spectral sequences} 
\label{swfs}

Till the end of this section $\cu$ will be endowed with a weight structure $w$; $H:\cu\to \au$ is a cohomological functor. For any $r\in \z$ denote $H\circ [-r]$ by $H^r$.

Let us recall some theory developed in (\S2 of) \cite{bws} and \cite{bger}. 

\begin{pr}\label{pwss}

1. For any $m\in \z$ the object $(W^{m}H)(M)=\imm (H(w_{\ge m}M)\to H(M))= \ke (H(M)\to H(w_{\le m-1}M))$
does not depend on the choice of $w_{\ge m}M$; it is $\cu$-functorial in $M$.

We call the filtration of $H(M)$ by $(W^{m}H)(M)$ its {\it weight} filtration.

2. For any $M\in \obj \cu$ there is a spectral sequence $T(H,M)$
with $E_1^{pq}=
H^{q}(M^{-p})$. It comes from (any possible) weight Postnikov tower of $M$; so the boundary of $E_1$ is obtained by applying $H^*$ to the corresponding choice of a weight complex for it. 

3. $T$ is (canonical and) naturally functorial in $H$ (with respect to exact functors between the target abelian categories) and in $M$ starting from $E_2$.

$T(H,M)$ converges to $H^{p+q}(M)$ if $M$ is bounded. Moreover, the  filtration step given by ($E_{\infty}^{l,m-l}:$ $l\ge k$)
 on $H^{m}(X)$ equals $(W^k H^{m})(M)$ (for any $k,m\in \z$). 

4. $T(H,M)$  satisfies all the properties 
of the previous assertions (for an arbitrary $M\in \obj \cu$) also in the case where $H$ kills $\cu_{w\le -i}$ and 
$\cu_{w\ge i}$ for some (large enough) $i\in \z$. 

\end{pr}

\begin{proof}

1. This is (a particular case of) Proposition 2.1.2(2) of ibid.

2,3.4. This is (most of) Theorem 2.4.2 of ibid.

\end{proof}

\begin{rema}\label{rintel}

1. Actually, $(W^{m}H)(M)=\imm (H(w_{\ge m}M)\to H(M))$
does not depend on the choice of $w_{\ge m}M$ and is functorial in $M$ for any contravariant $H:\cu\to \au$, 

2. Certainly, weight spectral sequences can also be constructed for a homological $H$; see Theorem 2.3.2 of ibid. 

\end{rema}

Now we  derive a simple (new) consequence from Proposition \ref{pwss}.

\begin{defi}\label{dpure}
We will call a cohomological functor $H:\cu \to \au$ (for an arbitrary abelian $\au$) a {\it pure} one if $H$ kills $\cu_{w\le -1}$ and $\cu_{w\ge 1}$.

\end{defi}

\begin{coro}\label{cwss}
Assume that $\cu$ is endowed with a weight structure $w$; let $\au$ be an abelian category. Then the restriction of functors to $\hw$ yields an equivalence of the "big category" of pure (cohomological) functors $\cu\to \au$ with the one of additive contravariant functors $\hw\to \au$. The converse correspondence is given by  Proposition \ref{pbw}(\ref{iwcoh}).

\end{coro}
\begin{proof}
The functors given by 
Proposition \ref{pbw}(\ref{iwcoh})  are pure by part \ref{iwc} of  the proposition; it also yields that the corresponding composition is equivalent to the identity of the big category of pure functors. The converse composition is equivalent to the identity of $\adfu(\hw^{op},\au)$ by 
Proposition \ref{pwss}(4). 

\end{proof}

\begin{rema}\label{rpure}
For any $\cu,H$ the functor (see Proposition \ref{pwss}(2)) $M\mapsto E_2^{0,0}T(H,M)$ is pure. Indeed, it suffices to verify that it is cohomological. The latter statement 
follows from the isomorphisms $E_2^{0,0}T(H,-)\cong \tau_{\le 0}(\tau_{\ge 0} H)\cong \tau_{\ge 0}(\tau_{\le 0} H)$ given by Theorem 2.4.2(II) of \cite{bger} (see Remark \ref{rdualn}(3) below and \S2.3 of ibid.). 

\end{rema}

\subsection{The relation to orthogonal $t$-structures}\label{sort}


Let $\cu,\du$  be triangulated categories.
We consider certain pairings 
$\cu^{op}\times \du\to \au$.
In the following definition we consider a general $\au$, yet
below we will mainly need $\au=\ab$. 

\begin{defi}\label{ddual}

1. We will call a (covariant) bi-functor
  $\Phi:\cu^{op}\times\du\to
\au$ a {\it duality} if  it is bi-additive, homological with respect
to both arguments, and is equipped with a (bi)natural bi-additive transformation
$\Phi(A,Y)\cong \Phi (A[1],Y[1])$.


2. We will say that $\Phi$ is {\it nice} if for any distinguished
triangles $T=A\stackrel{l}{\to} B \stackrel{m}{\to} C\stackrel{n}{\to} A[1]$ in $\cu$ and $X\stackrel{f}{\to} Y\stackrel{g}{\to} Z\stackrel{h}{\to}X[1]$ in $\du$ 
we have the following:
 the natural morphism $p$:
$$ \begin{gathered} 
 \ke (\Phi(A,X)\bigoplus \Phi(B,Y) \bigoplus \Phi(C,Z))\\
\xrightarrow{\begin{
pmatrix}
\Phi(A,-)(f) & -\Phi(-,Y)(l) &0  \\
0& g(B) &-\Phi(-,Z)(m)  \\
- \Phi(-,X)([-1](n)) & 0 &\Phi(C,-)(h)
\end{
pmatrix}}
\\ (\Phi(A,Y) \bigoplus \Phi(B,Z) \bigoplus \Phi(C[-1],X))
 \stackrel{p}{\to} \ke ((\Phi(A,X)\bigoplus \Phi(B,Y))\\ \xrightarrow{\Phi(A,-)(f)\oplus - \Phi(-,Y)(l)}
 \Phi(A,Y)) 
 \end{gathered}$$
is epimorphic.

 3. Suppose 
  $\cu$ is endowed with a weight structure $w$,
 $\du$ is endowed with a $t$-structure $t$. Then we will say that $w$
 is (left) {\it orthogonal} to $t$ with respect to $\Phi$
 if the following
 {\it orthogonality condition} is fulfilled:
 $\Phi (X,Y)=0$ if $X\in \cu_{w\ge 0}$
and $Y\in \du^{t \ge 1}$ or if $X\in \cu_{w\le 0}$ and $Y\in \du^{t \le -1}$.

\end{defi}

'Natural' dualities are nice; we will justify this thesis now.

\begin{pr}\label{pnice}
1. If $\au=\ab$, 
$F:\du\to \cu$ is an exact functor, then 
 $\Phi(X,Y)=\cu(X,F(Y))$ is nice.

2. For triangulated categories $\du$, $\cu'\subset \cu$, $\cu'$ is skeletally small,  
and $\au$ satisfying AB5,
let $\Phi':\cu'^{op}\times \du\to \au$ be a duality.
For any $Y\in \obj \du$ we extend the functor $\Phi'(-,Y)$ from
$\cu'$ to
$\cu$ by the method of Proposition \ref{pextc} below; we denote the functor obtained by
$\Phi(-,Y)$. Then the corresponding bi-functor $\Phi$ is a duality
($\cu^{op}\times \du\to \au$).
It is nice whenever $\Phi'$ is.

\end{pr}
\begin{proof}
1. It suffices to note that  the niceness restriction is a generalization 
of the axiom (TR3) of
triangulated categories
(any
commutative square can be completed to a morphism of distinguished
triangles) to the setting of dualities of triangulated categories. 

2. This is Proposition 2.5.6(3) of ibid.
\end{proof}

Now let us  describe the relation of weight spectral sequences
to orthogonal structures.

\begin{pr}\label{pdual}
Assume that $w$ on $\cu$ and $t$ on $\du$ are orthogonal with respect
 to a nice duality $\Phi$; 
 $M\in \obj \cu$,
  $Y\in \obj \du$. 


 Consider  the spectral sequence $S$  coming from
    the following exact couple: $D_2^{pq}(S)=\Phi(M,Y^{t\ge
q}[p-1])$, 
$E_2^{pq}(S)=\Phi(M,Y^{t=q}[p])$ (we start $S$ from $E_2$). 

This spectral sequence is naturally isomorphic to $T(H,M)$ for $H:\ N\mapsto
\Phi(N,Y)$ 
if we consider the latter as starting from $E_2$.

\end{pr}
\begin{proof}
This is (a part of) Theorem 2.6.1 of ibid.
\end{proof}

\section{Pro-spectra and the Gersten weight structure} 
\label{scomot}

We embed $\sh$ into a certain 
 triangulated motivic homotopy category
$\gdb$; we will call its objects {\it pro-spectra}. Below we will need several
properties of $\gdb$; yet our arguments usually  will 
not rely on its explicit definition.
For this reason, in \S\ref{comot} we only list the main properties of
$\gdb$ and of related categories. We use them in \S\ref{scgersten} for the construction of certain Gysin distinguished triangles and Postnikov towers for the pro-spectra of pro-schemes.


 Next we construct (in \S\ref{scwger}) 
  certain {\it Gersten
weight structures}  on  $\gdb\supset \gd$. $\gd$ is the subcategory of $\gdb$ cogenerated by $\shc$; one may say that it is the largest subcategory of $\gdb$ 'detected by the Gersten weight structure'.
The 
(common) heart of these Gersten weight structures is 'cogenerated' (via products and direct summands) by certain twists of pro-spectra of functions fields, whereas the pro-spectra of arbitrary pro-schemes belong to $\gd_{w\ge 0}$. It follows immediately that
the Postnikov tower $Po(X)$
  provided by 
  Proposition
\ref{post} is a
  {\it weight Postnikov tower} with respect to $w$.

  If $k$ is infinite, then $\hw$ also contains 
 the  pro-spectra of  semi-local pro-schemes over $k$.
   Using this fact, in \S\ref{stds}  we prove the following statement:
if $S$ is a semi-local pro-scheme (and $k$ is infinite),  $S_0$ is its dense sub-pro-scheme, then $\om(S_+)$ is a direct summand of $\om(S_{0+})$;
$\om(\spe K_+)$ (for a function field $K/k$)
 contains (as retracts)
    the pro-spectra of (affine essentially smooth) semi-local schemes whose generic point
is $\spe K$,
    as well as the twisted pro-spectra of residue
 fields of $K$ (for all geometric valuations). It follows 
 that the 'Gersten' weight complex  
of any semi-local pro-scheme splits.

\subsection {Pro-spectra: an 'axiomatic description'} 
\label{comot}

In \S\ref{scgd} below we will construct a triangulated category
 $\gdb$
   as the 
homotopy category of a certain stable model category  $\gdp$. 
In this section we  only describe certain properties of 
the categories $\gdb $ and $ \gdp$ that we will apply  somewhat 'axiomatically'. 


Let $\psh$ denote the model 
for $\sh$ constructed in \cite{jard} (we will say more on it  
in \S\ref{srsmc}  
below).

\begin{pr}\label{pprop}

\begin{enumerate}

\item\label{ifun}
There exists a commutative diagram of categories and functors
\begin{equation}\label{efun}
\begin{CD}
\opa@>{\pom}>>\psh  @>{c}>>\gdp  \\
@. @VV{\hosh}V@VV{\hogd}V \\
@.\sh  @>{\ho(c)}>>\gdb
\end{CD}\end{equation}
such that $\hosh\circ \pom=\om$. 


\item\label{ilim} $\gdp$ is closed with respect to all small filtered limits; $\gdb$ is closed with respect to all small products.

  \item\label{iemb} 
 $\ho(c)$ is a full exact embedding; it yields a family of  cocompact 
  cogenerators for $\gdb$.

We will often 
write $\sh\subset \gdb$ 
(without mentioning the embedding functor).

\item\label{icoclim} For any projective system 
$X_i\in \obj \gdp$, $C\in \obj \sh$ we have $\gdb(\hogd(\prli 
 X_i),C)\cong \inli_{i\in I} \gdb(\hogd(X_i),C)$. 

\item\label{itriang} Extend $\om$ to $\popa$ via the rule $\om(\prli P_i)=\hogd(\prli(c\circ \pom (P_i)))$.
 Then for any 
for any $j\ge 0$ and for any compatible system
of open embeddings  $Z_i\to Y_i\to X_i$ the natural morphisms $\om ((\prli Y_i/\prli Z_i\brj))\to \om ((\prli X_i/\prli Z_i\brj))\to \om((X_i/\prli Y_i \brj))$ extend to a distinguished triangle.  

\end{enumerate}

\end{pr}

We will usually call the objects of $\gdb$ {\it pro-spectra}.

We describe some 
 consequences of the 'axioms' listed.

\begin{coro}\label{css}

\begin{enumerate}

 
 \item \label{iisom}
 A $\gdb$-morphism $f:X\to Y$  is an isomorphism if and only if $\gdb(-,C)(f)$ is for any $C\in \obj \sh$.

\item \label{iisomil} In particular, for any projective system $I$ and any compatible system of morphisms $f_i\in \gdp(X_i,Y_i)$ we have the following fact: $\hogd(\prli f_i)$ is an isomorphism if and only if $\inli(\gdb(-,C)(\hogd(f_i))$ is an isomorphism for any $C\in \obj \sh$.  
 
 \item \label{icopr} Let $X=\sqcup X^\al$ be a decomposition of a pro-scheme. 
 Then for any $j\ge 0$ 
 the object $\om(X_+\brj)$ is naturally isomorphic to $\prod \om(X^\al_+\brj)$.

\end{enumerate}
\end{coro}
\begin{proof}
1. For $Z=\co f$ it suffices to note that (by Proposition \ref{pprop}(\ref{iemb})) we have $Z=0$ if and only if $Z\perp \sh$.

2. It suffices to combine the previous assertion with part \ref{icoclim} of the proposition.

3. The natural $\popa$-morphisms $X\to X^\al$ yield a canonical morphism $\om(X_+\brj)\to\prod \om(X^\al_+\brj)$. We have to check that this is an isomorphism. To this end (by assertion \ref{iisom}) we can fix a $C\in \obj \sh$ and verify for $H=\gdb(-,C)$ that $H(\om(X_+\brj)\cong H(\prod \om(X^\al_+\brj))$. Next, the cocompactness of $C$ in $\gdb$ yields that $H((\prod \om(X^\al_+\brj))\cong \bigoplus H(\om(X^\al_+\brj))$. It remains to recall that  $\sqcup X^\al_+\brj$ (for $\al\in A$) can be presented as the inverse  limit of  $\sqcup_{\al\in \be}X^\al_+\brj$ for $\be$ running through finite subsets of $A$ (in $\popa$; see \S\ref{sprs}); hence Proposition \ref{pprop}(\ref{icoclim}) yields the result.

\end{proof}


\subsection{%
The Gysin distinguished triangle
and 'Gersten' Postnikov towers for the pro-spectra of pro-schemes}\label{scgersten}

We 
make some more remarks on pro-schemes.

\begin{rema}\label{rpgysin}
For pro-schemes $U=\prli U_i$ and $V=\prli V_j$ we will call an
element of $\popa(U_+,V_+) $
 a closed (resp. open) 
 embedding if it can be obtained as the limit of closed (resp. open) 
 embeddings of pointed varieties (so, we consider only pro-smooth sub-pro-schemes of pro-schemes). 
 We define a general pro-embedding $U\to V$ similarly;  we will say that $U$ is of codimension $c$ in $V$
 if $U_i$ is of codimension $c$ in $V_i$ for any $i$; in particular, we will use this convention for defining the codimension of Zariski points.
 Similarly, we will  say that an inverse limit of open embeddings such that all complements have codimensions $\ge c$ is an open embedding of pro-schemes with complement of codimension $\ge c$.
 Certainly, the complement of a closed sub-pro-scheme of $V$ is always an open sub-pro-scheme. Also, the generic point of a connected pro-scheme is its open sub-pro-scheme. 
 
Now let us define normal bundles for closed embeddings of pro-schemes. For an embedding of a connected pro-scheme
$X=\prli X_i$ a normal bundle is an element of the direct limit of the sets of isomorphism classes of vector bundles over $X_i$. 
If a pro-scheme $X$ is actually a  scheme, then any closed embedding of $X$ into $Y$ does yield a (normal) 
 vector bundle over it; this is a projective module over the coordinate ring of $X$ if the latter is affine. Lastly, if $X$ is connected then the rank of this module is the codimension of $X$ in the corresponding component of $Y$, where the codimension of $X=\prli X_i$ in $X$ equals the one of any  $X_i$ that is connected.

Moreover, these observations are compatible with the isomorphisms given by Proposition \ref{psh}(\ref{ish5}). Besides, we can also pass to inverse limits for distinguished triangles given by part \ref{ish3} of the proposition if the connecting morphisms come from $\opa$; see Proposition \ref{pprop}(\ref{itriang}).
 
\end{rema}

\begin{pr}\label{pinfgy}

Let $Z,X$ be pro-schemes, $Z=\sqcup Z^\al$ ($Z^\al$ are connected) 
 be a closed 
 sub-pro-scheme 
 of $X$. Then the following statements are fulfilled. 
 
 1. For any
 $j\ge 0$ the natural morphism $\om(X\setminus Z_+)\to \om(X_+\brj)$
  extends to the following  distinguished triangle (in $\gdb$):
  $\om(X\setminus Z_+)\to \om(X_+\brj)\to \prod \om(N_{X,Z^\al}/N_{X,Z^\al}\setminus Z^\al \brj) $. 

2. Assume that all $Z^\al$ come from semi-local $k$-schemes (or just from $k$-schemes such that all vector bundles over them are trivial), and that $Z$ (and so, all $Z_\al$) are of codimension $c$ in $X$. Then the latter product converts into $\prod \om(Z^\al_+\lan j+c\ra )$. 
\end{pr}
\begin{proof}

1. As we have just noted, one can pass to the inverse limits of distinguished triangles given by Proposition \ref{psh}(\ref{ish3}). It remains to note that Corollary \ref{css} enables us to 
rewrite the third vertex of this triangle in the form desired.

2. If the normal bundles over all $Z^\al$ are trivial, then certainly $N_{X,Z^\al}/N_{X,Z^\al}\setminus Z^\al\brj \cong Z^\al_+\lan j+c\ra $ (in $\popa$). It suffices to note that this is the case for semi-local $Z^\al$ 
(see Remark \ref{rsemiloc}).

\end{proof}

\begin{rema}
1. It seems that instead of Corollary \ref{css} (which relies on $\sh$-cogenerators) we could have used a formal model category argument in the proof. 

2. The isomorphism of the second assertion is (usually) not canonical.
\end{rema}

Now let us construct a certain Postnikov tower $Po(M)$  for $M$
being the (twisted) pro-spectrum of a pro-scheme $Z$ that will be related
to the coniveau spectral sequences for (the cohomology of) $Z$. 
 Note that we consider the general case
of an arbitrary pro-scheme $Z$ (since in this paper pro-schemes play
an important role);  yet this case is not much distinct from the
(particular) case of $Z\in\sv$.

\begin{pr}\label{post}


Let $Z=\sqcup Z^\al$ be a pro-scheme of dimension $\le d$; for all $i\ge 0$ we denote by $Z^i$ the set of
points of $Z$ of codimension $i$.

Then any $j\ge 0$ there exists a Postnikov tower
for $M=\om(Z_+\brj)$ such that $l=j-1$, $m=d+j$, $M_{i+j}\cong \prod
_{z\in Z^i}\om(z_+\lan j+i\ra )$ for $0\le i\le d$.
\end{pr}
\begin{proof}

Since any product of distinguished triangles is distinguished (see Remark
1.2.2 of \cite{neebook}), we can assume $Z$ to be connected.

We consider a projective system $L$
whose elements are sequences of closed subschemes
$\varnothing=Z_{d+1}\subset Z_d\subset Z_{d-1}\subset \dots \subset
Z_0$. Here $Z_0\in\sv$, $Z_i\in \var$ for all $i>0$, $Z$ is (pro)-open in
$Z_0$,   
$Z_0$ is connected, and 
for all $i>0$ we have the following: 
(all irreducible
components of) all $Z_i$ are everywhere
  of codimension $\ge i$ in $Z_0$; 
 $Z_{i+1}$ contains the singular locus of $Z_i$ (for $i\le d$).
The  ordering in $L$ is given by open
 embeddings of varieties $U_i=Z_0\setminus Z_i$ for $i>0$.
For $\lambda\in L$ we will denote the corresponding sequence by
$\varnothing=Z^\lambda_{d+1}\subset Z^\lambda_d\subset Z^l_{d-1}\subset
\dots \subset Z^\lambda_0$. 

Now, for any $i\ge 0$ the limit $\prli_{\lambda\in L} Z^\lambda_i\setminus Z^\lambda_{i+1}$ equals $\sqcup_{z\in Z^i}z$. 
Hence the previous proposition yields a distinguished
triangle $\om (\prli(Z^\lambda_0\setminus Z^\lambda_i)_+\brj)\to
\om(\prli(Z^\lambda_0\setminus Z^\lambda_{i+1})_+\brj)\to \prod
_{z\in Z^i}\om(z_+\lan j+i\ra )$. So, setting $Y_{i+j}=\om (\prli(Z^\lambda_0\setminus Z^\lambda_{i+1})_+\brj)$ for $-1\le i\le d$, 
one obtains 
a tower as desired.

\end{proof}

\begin{rema}\label{lger}

1. The same reasoning also yields an unbounded positive Postnikov tower for $M=\om(Z_+)$ in the case where $Z$ is an arbitrary pro-scheme.

2. Certainly, if we shift a Postnikov tower for
 $\om(Z_+\brj)$ by $[c]$ for some $c\in \z$,
we obtain a Postnikov tower for $\om(Z_+\brj)[c]$. Yet if we want it to be a weight Postnikov tower with respect to the Gersten weight structure that we will construct below, we will also have to shift the indices by $c$.

\end{rema}


Now let us prepare to the construction of the Gersten weight structure.


\begin{pr}\label{pcoen}
Let $X$ be a pro-scheme, $j\ge 0$. Then $\om(X_+\brj)$ belongs to the coenvelope of $\om(U_+\brj)[i]$ for $U\in \sv$, $i\ge 0$.

\end{pr}
\begin{proof}

By Corollary \ref{css}(\ref{icopr}) it suffices to verify the statement for the connected components of $X$ (and for $j$ being fixed). Hence it is sufficient to prove 
that the result is valid for all pro-schemes of dimension $\le d$ by induction on $d$.

So, for some $d>0$ we can assume  that $\om(V_+\brj)$ belongs to the coenvelope of $\om(U_+\brj)[i]$ if $V$ is any pro-scheme of dimension $\le d-1$. It suffices to verify that the same is true for a fixed connected $X$ of dimension $d$.

Next, using Proposition \ref{post} (along  with Remark \ref{rpsh}(1)) one can easily see that $X$ can be replaced by its generic point, i.e., it suffices to verify the statement for $X$ being the generic point of some smooth (connected) $Z$ of dimension $d$. 

Now apply Proposition \ref{post}. 
We obtain distinguished triangles (see the notation of Definition \ref{dpoto}) $Y_{i+j-1}\to Y_{i+j}\to \prod
_{z\in Z^i}\om(z_+\lan j+i\ra )$ for $0\le i\le d$. Now, 
$Y_{d+j}=\om (Z_+\brj)$ belongs to the coenvelope in question, and the same is true for $\prod
_{z\in Z^i}\om(z_+\lan j+i\ra )$ for all $0<i\le d$ by the inductive assumption. Hence $Y_j=\om(X_+\brj)$ also belongs to this coenvelope. 
  
\end{proof}

\subsection{The 
Gersten weight structure: construction and basic properties}\label{scwger}

Now we describe the main weight structure of this paper.

We apply Theorem \ref{tnews} in the case $\cu=\gdb$, $C'=\{\om(X_+)[i]\} $ for $X\in \sv,\ i\ge 0$. 
Denote the corresponding category $\cu'$ by $\gd$; note that the full embedding $\ho(c):\sh\to \gdb$ restricts to an embedding $\shc\to \gd$ (we will just write $\shc\subset \gd$).
We obtain a weight structure $w$ on $\gd$. We will call it the {\it Gersten} weight structure, since it is
closely connected with Gersten resolutions of cohomology (cf.
Remark \ref{rdualn}(2) below). 
By default, 
below $w$ will denote
the Gersten weight structure.

Now, we  easily prove several important properties of this structure.

\begin{theo}\label{tgw}

The 
following statements are valid.
\begin{enumerate}

\item\label{iwg-1}
$\gd_{w\ge 0}$ is the coenvelope 
of $C'=\{\om(X_+)[i]\}$ for $X\in \sv,\ i\ge 0$, $\gd_{w\le 0}={}^\perp (C'[1])$.

\item\label{iwg0}
$\gd$, $\gd_{w\le 0}$, and $\gd_{w\ge 0}$ are closed with respect to all small products (both in $\gd$ and in $\gdb$).

\item\label{iwgprod} Small products of weight decompositions in $\gd$ are weight decompositions, and small products of weight Postnikov towers (resp. of weight complexes $t(M_i)$ for $M_i\in \obj \gd$) are weight Postnikov towers (resp. 
a choice for $t(\prod M_i)$).

\item\label{iwg1} 
$\om(X_+)\brj\in\gd_{w\ge j}$ for any pro-scheme $X$ and any $j\ge 0$. 

\item\label{iwg2} If $X$ is 
the spectrum of a function field over $k$, or if $X$ is semi-local and $k$ is infinite then  $\om(X_+\brj)\in\gd_{w=j}$. 

\item\label{iwg5}  For any pro-scheme $X$ and any $j\ge 0$ the Postnikov tower for $\om(X_+\brj)$ given by
Proposition \ref{post} is a weight Postnikov tower for it. In particular, if $X$ is of dimension $\le d$, then  $\om(X_+\brj)\in \gd_{[j,j+d]}$.

\item\label{iwgh} $\hw$ is 
equivalent to the Karoubization of the category of all $\prod \om(\spe K_{i+})\{ j_i \}$ for $K_i$ 
being function fields over $k$, $j_i\ge 0$.

\item\label{iwshcb} All objects of $\shc$ are bounded in $\gd$ (recall that we assume $\shc$ to be a subcategory of $\gd$).
 
\item\label{iwshc} We have $\obj\shc\cap \sh^{t\le 0}=\obj\shc\cap \gd_{w\ge 0}$.

\item\label{iwg3} If $f:U\to X$ is an open embedding of pro-schemes such that the complement is of codimension $\ge i$ in $X$ (see Remark \ref{rpgysin}), then $\om(X/U\brj)\in \gd_{w\ge i+j}$.

\item\label{iwg6} $\{\om(U_+)\}$ for $U\in \sv$ 
cogenerate $\gd$ (i.e., $\gd$ is the smallest triangulated subcategory of $\gdb$ that contains all $\{\om(U_+)\}$  and is closed with respect to all products); $w$ is non-degenerate from below.

\end{enumerate}

\end{theo}
\begin{proof}
\ref{iwg-1}. This is exactly the description of $w_{\cu'}$ given by Theorem \ref{tnews}.

\ref{iwg0}. The first part of the assertion is given by part II.1 of the theorem, the second one is given by part II.5 of the theorem, the third one is contained in Proposition \ref{pbw}(\ref{iwprod}).

\ref{iwgprod} Immediate from 
Theorem \ref{tnews}(II.6). 

\ref{iwg1}. For any $U\in \sv$ we have $\om(U_+\brj)\in \gd_{w\ge j}$ by Remark \ref{rpsh}(1). Hence the statement follows immediately from Proposition \ref{pcoen} (recall the definition of coenvelope).

\ref{iwg2}. It remains to verify that $\om(X_+\brj)\in \gd_{w\le j}$.
  By the definition of the latter class, for any $U\in \sv,\ i>j$ we should verify that $\om(X_+\brj)\perp \om(U_+)[i]$. 
  
  Now,  we have $\om(U_+)[i]\in \sh^{t\le -i}$ by Proposition \ref{psht}(1). Next,  Proposition \ref{pprop}(\ref{icoclim}) yields that for $E=\om(U_+)$ we have $\gd(\om(X_+\brj), \om(U_+)[i])\cong E^i_j(X)$ (in the notation used in Proposition \ref{pshinv}). Hence the proposition cited  yields the result in the case of an infinite $k$; if $k$ is finite (and $X$ is the spectrum of a function field) then one should apply Proposition \ref{psht}(3).


\ref{iwg5}. By the previous assertion, for this tower we have $M^i\in \gd_{w=0}$ (see Definition \ref{dpoto}(2)). Hence Proposition
\ref{pbw}(\ref{iwpostc}) yields the result. 

\ref{iwgh}. Immediate from the previous assertion by Theorem \ref{tnews}(III). 

\ref{iwshcb}. Immediate from the  previous two assertions.

\ref{iwshc}. We have $\gd_{w\le -1}\perp \gd_{w\ge 0}$. In particular, if $E\in \gd_{w\ge 0}$ then assertion \ref{iwg2} yields that for any function field $K/k$  we have $\om (\spe K_+)[-n]\perp E$ for any $n>0$. If $E$ also belongs to $\obj \shc$, this translates into 
$E^n(\spe K_{+})=0$ for all $n>0$ (see Proposition \ref{pprop}(\ref{icoclim})).
Hence $E\in \sh^{t\le 0}$ (see Proposition \ref{psht}(3)).

Now let us prove the converse implication. If $E\in \sh^{t\le 0}\cap \obj \shc$ then  $E_j^{n+j}(\spe K_+)=\ns$ for all $n>0$, $j\ge 0$, and function fields over $k$
 (see Proposition \ref{psht}(3)). Certainly this translates into $\om (\spe K_+\lan j\ra)[-n-j]\perp E$. Since $E$ is cocompact in $\gd$, 
 it follows (by assertion \ref{iwgh}) that $\hw[-n]\perp E$. Hence Proposition \ref{pbw}(\ref{icharge}) (along  with the previous assertion) yields that 
 $E\in \gd_{w\ge 0}$.

\ref{iwg3}. Since $f$ induces a bijection of the corresponding sets of points of codimension $<i$, it suffices to combine 
assertion \ref{iwg5} with part \ref{ipostn} of Proposition \ref{pbw}.


\ref{iwg6}. Immediate from Theorem \ref{tnews}(II.3).
\end{proof}

\begin{rema}

1. Describing weight
decompositions for arbitrary objects of
$\shc\subset \gd$ explicitly seems to be difficult. Still, one can say something about these weight
decompositions and
weight complexes using their functoriality properties.
In particular, knowing
weight complexes for $X,Y\in \obj \shc$
(or just $\in \obj \gd$) and $f\in \gd(X,Y)$ one can
describe the weight complex of $\co(f)$ up to a homotopy
equivalence as the
corresponding cone. 
 Besides,
let $X\to Y\to Z$
be a distinguished triangle (in $\gd$). Then for any choice of
 $(X_{w\le 0}, X_{w\ge 1})$ and $(Z_{w\le 0}, Z_{w\ge 1})$
there exists a choice
of $(Y_{w\le 0}, Y_{w\ge 1})$ such that there exist
distinguished triangles
 $X_{w\le 0}\to Y_{w\le 0}\to Z_{w\le 0}$ and
$X_{w\ge 1}\to Y_{w\ge 1}\to Z_{w\ge 1}$;
 see Lemma 1.5.4 of \cite{bws}.

2. The author suspects that $w$ is also non-degenerate from above. In any case, we will mostly be interested in bounded objects.

3. Certainly, we could also have considered the Gersten weight structure on the whole $\gdb$. Yet this does not seem to make much sense, since this weight structure will not 'detect' the objects of the corresponding 'extra summand' 
$\cuperp$; cf. 
Theorem \ref{tnews}(II.4). 

Besides, the proof of Theorem \ref{tnews}(II.2) provides us with a {\it weight-exact} 'projection'  $F:\gdb\to \gd$ (i.e., for $E\in \obj \sh$ 
we have $E\in  \gdb_{w\ge 0}$ if and only if $F(E)\in \gd_{w\ge 0}$ and $E\in  \gdb_{w\le 0}$ if and only if $F(E)\in \gd_{w\le 0}$). Moreover, if  $E\in \gdb_{w\ge 0}\cap \obj \sh$ then for any function field $K/k$ and $n> 0$ we have $E^n(\spe K_+)=\ns$; hence $E\in \sh^{t\le 0}$ (see Proposition \ref{psht}(3)). 
\end{rema}

\subsection{Direct summand results for  pro-spectra of semi-local schemes}\label{stds}

Theorem \ref{tgw} easily implies the following
interesting result.

\begin{theo}\label{tds1n}

1. Assume that $k$ is infinite; let $S$ be a semi-local pro-scheme and $S_0$ be its dense sub-pro-scheme.
Then $\om(S_+)$ is a direct summand of $\om(S_{0+})$.

2. Suppose moreover that $S_0=S\setminus T$, where  $T$ is a closed
sub-pro-scheme of $S$. 
Then we have $\om(S_{0+})\cong
\om(S_+) \bigoplus \om(N_{S, T}/N_{S, T}\setminus T)[-1]$.

3. Assume in addition that $T$ is of codimension $j$ (everywhere) in $S$. Then this decomposition takes the form
$\om(S_{0+})\cong
\om(S_+) \bigoplus \om(T_+\brj) [-1]$.

 \end{theo}
\begin{proof}

We can assume that $S$ and $S_0$ are connected.

1. By Theorem \ref{tgw}(\ref{iwg1}), we have $\om(S_{0+}),\om(S_+)\in
\gd_{w\ge 0}$. Besides, $\om(\spe k(S)_+)$ could be assumed to be the zeroth term
of their weight complexes for a choice of weight complexes compatible with some  positive Postnikov weight towers for them;
 the embedding $S_0\to S$ is compatible with
 $\id_{\om(\spe k(S)_+)}$ (since we have a commutative triangle $\spe k(S)\to S_0\to S$ of pro-schemes).
 Hence  Proposition \ref{pbw}(\ref{impostp}) yields the result.

2. By Proposition \ref{pinfgy}(1) we have a distinguished triangle
$\om(S_{0+})\to \om(S_+)\to \om(N_{S, T}/N_{S, T}\setminus T)[-1]$. By Theorem \ref{tgw}(\ref{iwg1},\ref{iwg2})
 we have $\om(S_{0+})\in \gd_{w\ge 0}$, $\om(S_+)\in \gd_{w= 0}$,
$\om(N_{S, T}/N_{S, T}\setminus T) \in \gd_{w\ge 1}$. Hence 
 Proposition \ref{pbw}(\ref{isump}) yields the result.

3. By Remark \ref{rsemiloc}, $T$ is semi-local itself.
Hence it suffices to combine the previous assertion with  Proposition  \ref{pinfgy}(2).


\end{proof}

\begin{coro}\label{tds1} \label{tds2}

I Assume that $k$ is infinite.

1. Let $S$ be a connected (affine essentially smooth) semi-local scheme; let $S_0$ be its generic
point.
Then $\om(S_+)$ is a retract of $\om(S_{0+})$.

2. Let $K$ be a function field  over $k$; let $K'$ be
the residue field for a geometric valuation $v$ of $K$ of rank $r$.
Then $\om(\spe K'_+)\{r\}$ is a retract
of $\om(\spe K_+)$.

II $\hw$ is equivalent to the Karoubization of the category of all  $\prod \om(\spe K_{i+})$ for $K_i$ being 
function fields over $k$. 
\end{coro}
\begin{proof}
 I.1. This is just a particular case of part 1 of the the theorem.

2. Obviously, it suffices to prove the statement in the case $r=1$.
Next,  $K$ is the function field  of some normal projective variety over
$k$. Hence there exists a $U\in \sv$ such that $k(U)=K$ and $v$ yields
a non-empty closed subscheme of $U$ of codimension $1$ everywhere. 
It easily follows that there exists a
pro-scheme $S$ (i.e., an inverse limit of smooth varieties) whose
only points are the spectra of $K$ and $K'$. So, $S$ is local affine; 
hence it is semi-local.

By part 2 of the previous theorem we have
$$\om(\spe K_+)=\om(S_+)\bigoplus \om(\spe K'_{+})\{ 1\};$$
this concludes the proof.

II If $k$ is infinite, 
 assertion I.2 yields that we can get rid of the twists mentioned in the (very similar) Theorem \ref{tgw}(\ref{iwgh}).

In $k$ is finite, one should apply the fact that $\om(\spe K'_+)\{r\}$ is a retract of $\om(\spe (k(G_m^r(K)))$ instead (whereas the statement mentioned can be easily established using the method of the proof of Theorem \ref{tds1n}(1)). 

\end{proof}

\begin{rema}\label{rstds} 
1. Note that we do not construct any explicit splitting morphisms in
the decompositions above. Probably, one cannot choose any canonical
splittings here (in the general case); so there is no (automatic)
compatibility for any pair of related decompositions. Respectively,
though the pro-spectra 
coming from  function fields contain tons of
direct summands, there seems to be no general way to decompose them
into indecomposable summands.

2. Yet Proposition \ref{pinfgy} easily yields that
$\om(\spe k(t)_+)\cong \om (\pt_+)\bigoplus \prod\om(z_+)\{ 1\}$; here $z$
runs through all closed points of $\af^1$ (considered as a scheme over $k$; note here that $\om(\af^1_+)\cong \om (\pt_+)$). 



\end{rema}


Now recall that any Postnikov weight tower for an $M\in \obj\gd$ defines an $\hw$-complex 
that is well defined  up to homotopy. It turns our that  the {\it augmented Gersten weight complex} for a semi-local pro-scheme 
splits.

\begin{pr}\label{psplger}
For a semi-local pro-scheme $S$ over an infinite $k$ consider the Postnikov tower of $M=\om (S_+)$
given by Proposition \ref{post}; denote the corresponding complex by $t(M)=M^i$. Then there exists some  $N^i\in \cu_{w=0},\ i\le 0$, such that 
$t(M)$  is $C(\hw)$-isomorphic to $M^0\bigoplus (\bigoplus_{i\le 0}(N^i\stackrel{\id_{N^i}}{\to}N^i)[-i])$.
\end{pr}
\begin{proof}
  Theorem \ref{tgw}(\ref{iwg5}) yields that $t(M)$ is a weight complex for $M$. Hence the result is given by Proposition \ref{pbw}(\ref{isplwc}).

\end{proof}

Below we will use this statement in order to deduce a similar result for cohomology.

\section{On cohomology and coniveau spectral sequences}  \label{sapcoh}

This is the central section of the paper.

In \S\ref{sextkrau} we describe (following H. Krause) a natural
method for extending cohomological
functors from a triangulated (small) $\cu'\subset\cu$ to $\cu$.
This method is compatible with the usual definition of cohomology for pro-schemes.

In \S\ref{sext} we  (easily) translate the results of the previous section to cohomology; in particular,
the cohomology of (the spectrum of) a function field
    $K/k$ contains direct summands
    corresponding to the cohomology of semi-local (affine essentially smooth) schemes whose
generic point is $K$,
    as well as twisted cohomology of residue
 fields of $K$. Note that 
 here  the cohomology of
 pro-schemes mentioned
 is calculated in the 'usual' way.

In \S\ref{sdconi} we consider  weight spectral sequences corresponding
to (the Gersten weight structure) $w$. We observe that these
spectral sequences generalize naturally  the classical coniveau spectral
sequences. Besides, for a fixed $H:\gd\to \au$ our (generalized) coniveau
spectral sequence converging to $H^*(M)$ (where $M$ could be any object of $\shc$  or a bounded object of $\gd$) is $\gd$-functorial in $M$ (in particular, 
it is $\shc$-functorial if restricted to $\shc$); this
 fact is non-trivial even when restricted to the spectra of smooth varieties.
Besides, we prove the following fact: any Noetherian subobject in the  image of 
any extended cohomology of a compact motivic spectrum belonging to $\sh^{t\le -r}$ (with respect to any $\shc$-morphism) in the cohomology of a smooth variety is supported in codimension $\ge r$ (for any $r> 0$).

In \S\ref{sconi} we construct a nice duality $\Phi:\gd^{op}\times \sh\to \ab$; we prove that $w$ is orthogonal to $t$ with respect to $\Phi$. It follows that  our {\it generalized coniveau} spectral sequences can be expressed in terms of $t$ (starting from $E_2$); this vastly generalizes the corresponding seminal result of \cite{blog}.

In \S\ref{sat} we note that for certain varieties and spectra one can choose quite 'economical' 
versions of weight Postnikov towers and (hence) of the (generalized) coniveau spectral sequences for cohomology.

In \S\ref{spure} we prove that all {\it pure extended} cohomological functors $\gd\to \au$ 
come (via the correspondence provided by Corollary \ref{cwss}) from those (contravariant additive)
functors $\hw\to \au$ that convert all products into coproducts. As a consequence, we prove (if $k$ is infinite)
  that  
  this result (applied for $\au=\ab$) yields a complete description of the heart of $t$ (i.e., of the category of strictly homotopy invariant Nisnevich sheaves on $\sm$).

\subsection{Extending cohomology from $\shc$ to $\gd$}\label{sextkrau}

Certainly, we would like to apply the results of the previous sections to the cohomology of pro-spectra. The problem is that cohomology is 'usually' defined on $\sh$ (or on $\shc$). So
we describe (and apply) a general method for extending cohomological
functors from a full triangulated $\cu'\subset\cu$ to $\cu$
(after H. Krause). Its advantage is that it  yields functors that are 'continuous' with respect to inverse limits in $\gdp$.

The construction requires $\cu'$ to be
skeletally small, i.e., there should exist a 
 subset (not just a subclass!) $D\subset \obj \cu'$ such that any object of $\cu'$ is isomorphic to some element of $D$; this is certainly true for $\shc$.
Since the distinction between small and skeletally small categories will not affect our  arguments and results,  we will ignore it in the rest of the paper.

Recall that for an abelian category  $\au$ and any small $\cu'$ the category $\adfu(\cupr^{op},\au)$
 is abelian also; complexes
in it are exact if and only if they are exact when applied to any object of $\cu'$, and the same is true for coproducts.


Now suppose that $\au$ satisfies AB5. 

\begin{pr}\label{pextc}
I Let $\au$ be fixed; consider an $H'\in \adfu(\cupr^{op},\au)$. 

1. One can construct an  extension of $H'$ 
to an additive functor $H:\cu\to \au$ (i.e., the restriction of $H$ to $\cu'$ is equal to $H'$). It is cohomological
 if and only if $H$ is. The correspondence $H'\mapsto H$ given by this construction is functorial and additive in the obvious sense. 

2. Moreover, suppose that in $\cu$ we have a projective
system $X_l,\ l\in L$, equipped with a compatible system of
morphisms $X\to X_l$, such that the latter system for any
$Y\in \obj \cupr$ induces an isomorphism
$\cu(X,Y)\cong \inli \cu(X_l,Y)$. Then we have $H(X)\cong \inli H(X_l)$.

II Let $X\in \obj \cu$ be fixed.

1. For a family of $X_l\in \obj \cu'$ and
$f_l\in \cu(X,X_l)$ assume that $(f_l)$ induce a surjection
$\bigoplus \cu(X_l,-)\to \cu(X,-)$ in $\adfu(\cu',\ab)$. Then 
$f_l$ also yield a surjection $\bigoplus H'(X_l)\to H(X)$. 

Moreover, such a set of $(X_l,f_l)$ exists for any $X\in \obj \cu$.


2. Let $F'\stackrel{f'}{\to} G' \stackrel{g'}{\to} H'$
be a (three-term) complex in $\adfu(\cupr^{op},\au)$ that is exact in
the middle; suppose that $H'$ is cohomological. Then the
complex $F\stackrel{f}{\to} G \stackrel{g}{\to} H$
(here $F,G,H,f,g$ are the corresponding extensions) is exact
 in the middle also.

3. If $H'\cong \coprod H_i'$ in $\adfu(\cupr^{op},\au)$, then $H(X)\cong \coprod H_i(X)$ for the corresponding extension functors $H_i$.

III Apply the previous assertions for $\cu=\gd$, $\cu'=\shc$ (note that $\shc$ is skeletally small). Then 
the extension of  $H':\shc\to \au$ to
$H:\gd\to \au$ satisfies the following properties. 

1. $H$ converts  those inverse
limits  in $\gdp$ that are mapped by $\ho(c)$ inside $\gd$ into the corresponding direct limits in $\au$. 

2. $H$ converts products in $\gd$ into coproducts in $\au$.

3. $H$ 
can be characterized (up to a canonical isomorphism) as the only extension of $H'$ that 
satisfies 
the  previous two assumptions.

4. $H$ converts countable homotopy limits in $\gd$ into the corresponding direct limits in $\au$.


\end{pr}
\begin{proof}
Assertions I, II are simple  applications of the results of \cite{krause}; see Proposition 1.2.1 of \cite{bger}. 

III 
These statements are easy consequences of assertion I.2.

In order to obtain assertion III.1 this result should be combined with   
Proposition \ref{pprop}(\ref{icoclim}). 
Next, part \ref{iemb} of the proposition (the cocompactness of objects of $\shc$ in $\gd$) 
immediately yields assertion III.2. 

Assertion III.3 is given by (the dual to) Lemma 2.3 of \cite{krause} (note that $\shc$ cogenerates $\gd$; see Theorem \ref{tgw}(\ref{iwg6})).

In order to obtain assertion III.4 
 one should 
 apply assertion III.2 and recall Lemma \ref{coulim}(3).
\end{proof}

\begin{rema}\label{rcohp}

1.   In the setting of assertion III we will call  $H$ an
{\it extended} cohomology theory. Note that assertion III.3 yields a complete characterization of extended theories.

2. As a particular case of assertion III.1,  for any $j\ge 0$ and any pro-scheme $X=\inli X_i$ we have $H(\om(X_+\brj))\cong \inli H(\om(X_{i+}\brj))$. This is certainly compatible with the usual way 
 of extending cohomology from varieties to their inverse limits; so  all of the results below can be applied to the 'classical' definitions of $K$-theory, algebraic cobordism, motivic cohomology, etc. of semi-local 
 schemes. For example, the value of an extended theory at the 
 spectrum of an (essentially smooth) local (resp. Henselian) ring will be the corresponding Zariski (resp. Nisnevich) residue.

Also, recall that $H$ coincides with $H'$ on $\shc$ (see assertion I.1); 
hence we obtain 'classical' values of cohomology for all compact objects of $\sh$ also.


3. We will construct $\gdp$ in \S\ref{scgd} below as the category of (filtered) pro-objects of $\psh$. Hence (by part I.2 of the Proposition) all extended cohomology theories factor through the category $\proo - \sh$ of 'naive' pro-objects of $\sh$ (that is certainly not triangulated). 
Thus  the pairing $\Phi:\gd^{op}\times \sh$ that we will construct in \S\ref{sconi} below factors through $(\proo - \sh)^{op}\times \sh$.

4. Certainly, the direct summand results of Theorem \ref{tds1n} also yield  similar statements  in $\proo-\shc\subset\proo-\sh$.

\end{rema}

\subsection{On cohomology
of pro-schemes, and its direct summands}\label{sext}

We easily prove that the results of the previous section
easily imply 
similar assertions for extended cohomology theories. 

\begin{pr}\label{cdscoh}
Suppose 
$k$ is infinite.

Let  $S$ be a semi-local pro-scheme; let $H:\gd\to \au$ be an extended cohomology theory. 

1. Let $S_0$ be a dense sub-pro-scheme of  $S$. Then $H(\om(S_+))$
  is a direct summand of $H(\om(S_{0+}))$.

2. Suppose moreover that $S_0=S\setminus Z$, where  $Z$ is a
closed subscheme of $S$ of codimension $j>0$. Then we have
$H(\om(S_{0+}))\cong H(\om(S_{+}))
\bigoplus H(\om(Z_+\brj)[-1])$.

3. Suppose 
 $S$ is connected and 
  $S_0$ is the generic point of $S$.
Then $H(\om(S_+))$ is a retract of $H(\om(S_{0+}))$ in $\au$.

4. Let $K$ be a function field  over $k$. Let $K'$
be the residue field for a geometric valuation $v$ of $K$ of rank
$j$.
Then $H(\om(\spe K'_+)\{j\})$ is a retract of $H(\om(\spe K_+))$ in $\au$.

5. 
Consider the Postnikov tower of $M=\om (S_+)$
given by Proposition \ref{post}; denote the corresponding complex by $t(M)=M^i$. 
Denote by $T_H(S)$ the {\it Cousin} 
 complex 
$(H(M^{-i}))$. Then there exist some $A^i\in \obj \au$ for $i\ge 0$ such that $T_H(S)$ is $C(\au)$-isomorphic to $H(S)\bigoplus A^0\to A^0\bigoplus A^1\to A^1\bigoplus A^2\to \dots$.

\end{pr}
\begin{proof}

1. Immediate  from
Theorem \ref{tds1n}(1).

2. Immediate  from
 Theorem \ref{tds1n}(2).

3. Immediate  from
 Corollary \ref{tds1}(I.1).

4. Immediate  from
 Corollary \ref{tds2}(I.2).
 
 5. Immediate from Proposition \ref{psplger}.

\end{proof}

\begin{rema}\label{runiv}
Certainly, assertion 5 for $H$ being a 'motivic homotopy' theory is stronger then the universal exactness Theorem 6.2.1 of \cite{suger}. A caution: the definition of the universal exactness given in ibid. is not quite correct; see \cite{suzain}.
\end{rema}

\subsection{Coniveau spectral sequences for the cohomology of (pro)spectra}
\label{sdconi}

Let $H:\gd\to \au$ be a cohomological functor, $M\in \obj \gd$.


\begin{pr} \label{rwss}

I.1. Any choice of a weight spectral sequence $T(H,M)$
(see Proposition \ref{pwss}) corresponding to the Gersten weight
 structure $w$ is canonical and $\gd$-functorial in
 $X$ starting from $E_2$.

2. $T(H,M)$ converges to $H(M)$ if $M$ is bounded with respect to $w$.

3. Let $H$ be an extended theory (see Remark \ref{rcohp}),
 $M=\om(Z_+)$ for a $Z\in \sv$.
Then any choice of $T(H,M)$ starting from $E_2$ is canonically
isomorphic to the classical coniveau spectral sequence (converging
to the $H$-cohomology of $Z$; see  \S1 of \cite{suger}). 
In particular, the corresponding filtration is the coniveau one.

II.  Let $M'\in \obj \shc\cap \sh^{t\le -r}$ for some $r\in \z$. Then the following statements are valid.

1. $H(M')=(W^{r}(H))(M')$ (see Remark \ref{rintel}(1)).

2. For any $g\in \gd(M,M')$ we have $\imm(H(g))\subset (W^{r}(H))(M)$.

3. Now let $H$ be an extended theory,
 $M=\om(Z_+)$ for a $Z\in \sv$. For a $g$ as above 
consider a Noetherian subobject $A$ of  $\imm(H(g))$ (i.e., we assume that any ascending chain of subobjects of $A$ in $\au$ becomes stationary). Then $A$  is {\it supported in codimension $r$}, i.e.,
 there exists an open $U\subset Z$ such that $Z\setminus U$ is of codimension $\ge r$ in $Z$ 
and $A$ is killed by the restriction morphism $H(\om(Z_+))\to H(\om(U_+))$.  

\end{pr}
\begin{proof}
I.1. This is just a particular case of  
Proposition \ref{pwss}.

2. Immediate since $w$ is bounded;
see part 3 of the proposition.

 3. Recall that in Proposition \ref{post}  a 
   Postnikov tower $Po(M)$ for $M$
 was obtained from certain 'geometric' Postnikov towers
by passing to the inverse limit in $\gdp$.

Now, in 
\S3 of \cite{ndegl}  two exact couples (for the $H$-cohomology of $Z$) were constructed. 
One of them  was obtained by applying $H$ to 
our geometric towers 
and then passing to the inductive limit (in $\au$).
Moreover, it was shown that this couple yields 
the same (coniveau) spectral sequence as the other one mentioned in loc. cit. (see \S2.1 of ibid.; cf. also Remark 2.4.1 of \cite{bws}), whereas the latter couple coincides with the one considered in  \S1.2 of \cite{suger}. 
 Furthermore,
  Remark \ref{rcohp}(2) yields that the  limit mentioned is
 (naturally) isomorphic to 
the spectral sequence obtained via  $H$ from $Po(M)$. 
Next, since $Po(M)$ is a weight Postnikov tower for $M$, 
 the latter spectral
 sequence  is one of the possible choices for $T(H,M)$.

 Lastly, assertion 1 yields that all other possible
$T(H,M)$ (they depend on the choice of a weight
 Postnikov tower for $X$) starting from $E_2$ are also
canonically isomorphic to the classical coniveau spectral
sequence mentioned.

II.1. By Theorem \ref{tgw}(\ref{iwshc}) we have $M'\in \gd_{w\ge r}$. Hence we can take $w_{\ge r}M'=M'$, and the result is immediate from Remark \ref{rintel}(1).

2. Immediate from the $\gd$-functoriality of our weight filtration (given by the remark) along  with the previous assertion.

3. According to the previous  assertion,
$A$ lies in $(W^r(H))(M)$.
By  assertion I.3 
this means that $A$ dies in $\inli_{U\subset Z,\ \codim_Z(Z\setminus U)\ge r}H(\om(U_+))$. Since $A$ is noetherian, it also vanishes in some particular $U$ of this sort.
\end{proof}

\begin{rema}\label{rrwss}

1. Assertion II.3 along  with its motivic analogue (see \S\ref{sdm} below) could be quite actual for the study of 'classical' motives. Note that one can apply it for $M'$ being a 
cone of some morphism $\om(Z'_+)\to \om(Z_+)$ for some $Z\in \sv$ (here $Z$ could be a point),  
whereas $H$ can be the $i$-th cohomology for some 'standard' cohomology theory  and an  $i\in \z$.

Besides, one can easily prove 
the assertion for any (not necessarily additive) functor $\gd\to \au$ that converts homotopy limits 
into direct limits and sends zero morphisms into zero maps.

Certainly, the statement is interesting only if $r>0$.
 
2. Assertion I.3 of the proposition yields a good reason to
call (any choice of)   $T(H,M)$ a
{\it generalized coniveau spectral
sequence} (for a general $H,\au$, and $M\in \obj \gd$); this
will also distinguish (this version of) $T$ from weight spectral
 sequences corresponding to other weight structures. We will
 give more justification for this term in Remark \ref{rconiv} below.
  So, the corresponding filtration
can be called the (generalized) coniveau filtration (for a general $M$).

3. 
Actually, in order to obtain a coniveau spectral sequence for $(H,Z)$
using the recipe of 
\cite{deggenmot} and \cite{suger} it is not sufficient to 
compute just the cohomology of (the spectra of smooth) varieties. One also needs to apply $H$ to certain objects of $\opa$
in order to compute the $E_1$-terms of the exact couple, whereas the connecting morphisms of the couple come from the natural comparison morphisms between relative cohomology and the cohomology of varieties 
(see \S1.1 and Definition 5.1.1(a) of ibid.). So, for those 'classical' cohomology theories for which 
 all of this information has an 'independent' definition, one should check whether it
 can be 'factored through $SH$'. This seems to be true for all of 'well-known' theories. 
For $K$-theory this fact is given by Corollary 1.3.6 of \cite{papi}; for \'etale cohomology the proof is easy; one could use an argument from the proof of Theorem 
4.1 of \cite{ndegl}.


On the other hand, in order to calculate the coniveau filtration it suffices to know the restriction of $H$ to the (spectra of smooth) varieties; so this does not require any of this complicated extra information.
Besides, our (pretty standard) arguments yield that for any $i\in\z,\ j\ge 0$ and any 
cohomology theory satisfying axioms 5.1.1(a), COH1, and COH3 of \cite{suger} (which is certainly the case for all of the examples we are interested in) the $E_1$-terms of the 'standard' coniveau spectral sequences are isomorphic to our ones.

\end{rema}

\subsection{A duality of motivic spectral categories; 
comparing  spectral sequences}\label{sconi}\label{sdual}

In order to apply the formalism of orthogonal structures we need the following statement.

\begin{pr}\label{pdualsh}
For each $M\in \obj \sh$ consider the (cohomological) functor $H_M:\gd\to \ab$ obtained by extending $\gd(-,M)$ 
via Proposition \ref{pextc}(III). 

Then the following statements are valid.

I.1. The  pairing $\Phi:\gd^{op}\times \sh\to \ab:\ \Phi(X,M)=H_M(X)$ is a nice duality of triangulated categories.

2. For any $X\in \obj \gd$ the functor $\Phi(X,-)$ converts coproducts into coproducts.

II The Gersten weight structure $w$  on $\gd$ is orthogonal to the homotopy $t$-structure $t$ on $\sh$. 

III For any $M\in \obj \sh$ the functor $\Phi(-,M)$ converts  filtered inverse limits in $\gdp$ into direct
 limits in $\ab$. 

\end{pr}
\begin{proof}
I.1. Immediate from Proposition \ref{pnice}.

I.2. Immediate from Proposition \ref{pextc}(III.2).

II We should check that 
 $\Phi (X,Y)=\ns$ if either (1) $X\in \gd_{w\ge 0}$
and $Y\in \sh^{t \ge 1}$ or if (2) $X\in \gd_{w\le 0}$ and $Y\in \sh^{t \le -1}$.

We verify 
the orthogonality in question in the setting (1). Proposition \ref{pextc}(III.4) yields that it suffices to present $X$ as a countable homotopy limit of certain $X_i\in \obj \gd$
such that $\Phi(X_i,Y)=\ns$. Since $\Phi(-,Y)$ is a cohomological functor, Theorem \ref{tgw}(\ref{iwg-1}) yields the following: 
it suffices to verify that $\Phi(\prod \om(S_{i+})[n_i],Y)=\ns$ for any family of $S_i\in \sv$, $n_i\ge 0$. 
Next, Proposition \ref{pextc}(III.2) yields that it suffices to verify this statement for a single $S_i$. In this case the assertion is immediate from the definition of  $t$ (see 
Proposition \ref{psht}(2)) and the fact that the restriction of $\Phi(-,Y)$ to $\shc$ (by definition) is just $\sh(-,Y)$.

Lastly let us verify the orthogonality assertion for the setting (2). 
Since $\Phi(X,-)$ is homological and respects coproducts, Proposition \ref{psht}(5) yields the following: it suffices to verify that 
$\Phi(X, \om(S_{i+}) [n_i])=\ns$ for any $S_i\in \sv$, $n_i\ge 1$. 
The latter follows from Proposition \ref{pextc}(II.1) along  with the orthogonality axiom for $w$.

III Immediate from Proposition \ref{pextc}(III.1).
\end{proof}

\begin{rema}\label{rhocolim}
1. Suppose 
  we have an inductive system $M_i\in \obj\sh$, $i\in I$, connected
by a compatible family of morphisms with some $M\in \sh$  that satisfy the following condition:
for any $Z\in \obj \shc$ 
we have $\sh(Z,M)\cong\inli \sh(Z,M_i)$
(via these morphisms $M_i\to M)$. In such a situation it is
reasonable to call $M$ the homotopy colimit of $M_i$. Note that all (small) coproducts are examples of homotopy colimits of this sort (by the definition of compact objects).

Then Proposition \ref{pextc}(I.2) yields that for
$X\in \obj \gd$ we have $\Phi(X,M)=\inli \Phi(X,M_i)$. So, one may
say  that all objects of $\gd$ are 'compact with respect to $\Phi$',
whereas part 3 of the proposition yields that all objects of $\sh$
are 'cocompact with respect to $\Phi$'. Note that both of these  properties 
 fail for the duality $\Phi'(X,M)=\gdb(X,M)$. 

2. Now we continue pursuing the issue of "localizing scalars"; see Remark \ref{rrcoeff}(1). So, $S\subset \z$ is a set of primes, and we consider $\gd[S\ob]$.

Next we note (dually to the remark cited and following  Appendix A.2 of \cite{kellyth} as well as Appendix B of \cite{levconv}) the following: since $\sh$ is compactly generated, one can construct a certain compactly generated category 
$\sh[S\ob]$ by localizing $\sh$ by  cones of $c\stackrel{\times s}{\to} c$ for $c\in \obj \sh,\ s\in S$. The localization functor $l:M\mapsto M[S\ob]$ respects (small) coproducts and compact objects; for  any $c\in \obj  \shc$ and  $c'\in \obj \sh$ we have $\sh[S\ob](c[S\ob],c' [S\ob])\cong \cu(c,c')\otimes_\z \z[S\ob]$.

This formula easily implies that 
 the homotopy $t$-structure for $\sh$ "induces" a homotopy $t$-structure on $\sh[S\ob]$. Indeed, we note that $t$ has a "description in terms of compact objects" (see Proposition \ref{psht}(3)); thus it suffices to apply the general $t$-structure results of (the Appendix of) \cite{talosa}.

Next, the "bicontinuity" of $\Phi$ (as described above) easily yields that  the group $\Phi(X,M)$ is $S$-torsion (i.e., any of its elements is killed by some product of elements of $S$) if $X$ or $M$ is a "generalized $S$-torsion object" (i.e., if $X$ is an object of the corresponding $\gd_{S-tors}$ or if $M$ belongs to $\sh_{S-tors}$ defined as the corresponding localizing subcategory of $\sh$).  Hence $\Phi':(X,M)\mapsto \Phi(X,M)\otimes_\z \z[S\ob]$ induces a well-defined duality $\Phi[S\ob]$ of  $\gd[S\ob]$ with $\sh[S\ob]$.  
Since $w$ is orthogonal to $t$, the corresponding $w[S\ob]$ for $\gd[S\ob]$  is obviously orthogonal to $t[S\ob]$ on $\sh[S\ob]$. 

Certainly, these remarks can also be applied to the categories $\gdt$ and $\sht$ that we will consider below (and we will need $S=\{2\}$ and $S=\p$; the latter choice gives certain $\q$-linear categories).

\end{rema}

Now let us  relate generalized coniveau spectral sequences to the
homotopy $t$-structure (in $\sh$). This is a
vast extension of the
seminal results of \S6 of \cite{blog} (i.e., of the calculation
by Bloch and
Ogus of the $E_2$-terms of coniveau spectral sequences)
and of \S4 of \cite{ndegl}.

\begin{coro}\label{ccompss}

 If
$H$ is represented by a $Y\in \obj\sh$ (via our $\Phi$) then for a motivic
(pro)spectrum $M$ our generalized coniveau spectral sequence $T(H,M)$ starting from
$E_2$ can be naturally and $\gd^{op}\times \sh$-functorially expressed in terms of the cohomology of $M$
with coefficients in the
$t$-truncations of $Y$ (as in Proposition \ref{pdual}). 


\end{coro}

\begin{proof}

 Immediate from Proposition \ref{pdualsh}.

\end{proof}

\begin{rema}\label{rdualn} \label{rconiv} 

1. Our
comparison statement is  true for the $Y$-cohomology of an
arbitrary $M\in \obj \shc$; this  extends to $\shc$ Theorem
4.1 of \cite{ndegl} (whereas the latter essentially extends the
results of \S6 of \cite{blog}). We obtain one more reason  to call
$T$ (in this case) a generalized coniveau spectral sequence for
(the cohomology of) motivic spectra.

Note also that the methods of Deglise do not (seem to) yield the $\shc$-functoriality of the isomorphism in question. 

2. If $Y\in \sh^{t=0}$ (i.e., it is a {\it strictly 
$\afo$-invariant} Nisnevich sheaf with transfers; see Lemma 4.3.7 of \cite{morintao} or the proof of Corollary \ref{coextpure} below), then $E_2(T)$ yields the
Gersten resolution
for the sheaf $H(-_+)$  (when $X$ varies); this is why we called $w$ the
Gersten weight structure.

3. In \S2.3--2.5 of \cite{bger} for any weight structure $w$ on a triangulated category $\cu$, and a cohomological $H:\cu\to \au$ certain {\it virtual $t$-truncations} were studied. For any $M\in \obj \cu$, $i\in \z$, one considers $\tau_{\le i}H(M)=\imm (H(w_{\ge -i-1})\to H(w_{\ge -i})) $, $\tau_{\ge i}H(M)=\imm (H(w_{\le -i})\to H(w_{\le -i+1})) $ (the connecting morphisms here are uniquely determined by the choices of the corresponding weight decompositions). %
Virtual $t$-truncations have several nice properties; in particular, we obtain cohomological functors $\cu\to\au$. The name is justified by the following isomorphisms of functors from $\gd$ to $\ab$: 
$\Phi(-,(t^{\le i}Y)[-i])\cong \tau_{\le i}(\Phi(-,Y))$ and $\Phi(-,(t^{\ge i}Y)[-i])\cong \tau_{\ge i}(\Phi(-,Y))$; see Proposition 2.5.4(1) of \cite{bger}.

4. A related observation is that virtual $t$-truncations of arbitrary extended cohomology theories are also extended (see Theorem \ref{tgw}(\ref{iwgprod}) and Remark \ref{rcohp}(1)). 

\end{rema}

\subsection{'Simple' coniveau spectral sequences for the cohomology of Artin-Tate spectra}\label{sat}

The formalism of weight structures yields much flexibility for the calculation of $T(H,M)$ (that certainly yield spectral sequences that are canonically isomorphic starting from $E_2$). Usually none of these spectral sequences are 'simple'; yet for certain $M\in \obj \shc\subset \obj \gd$ there exist very 'economical' weight Postnikov towers.

Indeed, assume that $M$ is an {\it Artin-Tate spectrum}, i.e.,  there exist finite extensions $k_i$ of $k$ such that $M$ belongs to the triangulated subcategory of $\shc$ generated by 
$\om (\spe k_i{}_+\lan j_i\ra)$ (for $j_i\ge 0$; certainly, the number of $k_i$ can be assumed to be finite here). 
We note that there exist a weight structure for Artin-Tate spectra whose heart consists of (finite!) direct sums of objects of this sort (and their retracts; see Remark \ref{rwgen}(2)), whereas a weight Postnikov tower with respect to this weight structure is certainly also a Gersten weight structure (by Proposition \ref{pbw}(\ref{iwpostc})).

One can construct a vast family of varieties that yield Artin-Tate spectra. To this end it is useful to note that this category contains $\om (\spe k_i{}_+\lan j_i\ra)$, $\om (\gmmpl)$, $\om (\p^n_+)$ for any $n\ge 0$, and it is a tensor triangulated subcategory of $\shc$ (whereas the tensor product corresponds to the product of varieties). Besides, if one replaces $\sh$ by $\ho(\mglmod)$ here (see \S\ref{smgl} 
below) then can apply the 'usual' Gysin distinguished triangle in order to find pro-schemes whose pro-spectra are Artin-Tate ones.

Lastly, one can detect Artin-Tate spectra using the weight complex functor (cf. Proposition \ref{pbw}(\ref{iwc})); see 
Corollary 8.1.2 of \cite{bws}.

\subsection{On pure extended  cohomology theories} 
 \label{spure} 

Corollary \ref{cwss} gives a certain description of the 'big category' of pure functors $\gd\to \au$. Yet it is certainly more interesting to analyse the category of pure extended functors (when $\au$ is an AB5-category) instead since these functors 
essentially  yield a full subcategory of  $\adfu(\shc,\au)$. 

It turns out that  it is quite easy to distinguish pure extended functors from all other pure ones; the corresponding restriction is somewhat similar to the one for Brown representability.

\begin{pr}\label{pextpure}
Suppose 
 $\au$ satisfies AB5.

Then restricting pure (see Definition \ref{dpure}) extended functors $\gd\to \au$ to  $\hw$
yields an equivalence of the category of these functors with the category of those additive contravariant functors from
$\hw$ to $\au$ that satisfy the following {\it cocompactness} condition: they convert all (small) $\hw$-products into $\au$-coproducts.


\end{pr}
\begin{proof}
We only have to check that the restriction of the equivalence provided by Corollary \ref{cwss} to  pure extended functors yields the equivalence of this category with the one of cocompact functors from $\hw$ to $\ab$.

By Proposition \ref{pextc}(III.2), extended functors convert $\hw$-products into coproducts.
Moreover, part III.3 of the proposition yields the following: it remains to prove that the functor obtained from a cocompact $H_0:\hw^{op}\to \au$ (by the method given by Proposition \ref{pbw}(\ref{iwcoh})) yields a functor that converts $\gd$-products into $\au$-coproducts. The latter is immediate from Theorem \ref{tgw}(\ref{iwgprod}).

\end{proof}

\begin{rema}
Note that the {\it virtual $t$-cohomology} of any extended cohomological functor is pure extended; cf. Remarks \ref{rconiv}(3,4) and 
\ref{rpure}.
\end{rema}

Applying our proposition we obtain a certain description of the heart of $t$, i.e., of the category of strictly homotopy invariant sheaves (see Definition 4.3.5 of \cite{morintao}).

\begin{coro}\label{coextpure}
Suppose 
$k$ is infinite.

Then the restriction of $\Phi:\gd^{op}\times \sh\to \ab$ to $\hw^{op}\times \hrt$
yields an equivalence of $\hrt$ with the category of cocompact (additive contravariant) functors $\hw\to \ab$.
\end{coro}
\begin{proof}
The previous proposition yields the following: the restriction of $\Phi$ mentioned gives an equivalence of the category of cocompact functors $\hw\to \ab$ with the one of pure extended functors from $\gd$ to $\ab$. Applying 
Proposition \ref{pextc}(III.3) we obtain the following: it suffices to prove that $\sh(-,-)$ fully embeds $\hrt$ into the category of cohomological functors from $\shc$ to $\ab$, and $\hrt$-representable functors are 
exactly those ones whose extensions to $\gd$ are pure. The first statement mentioned is an immediate consequence of Lemma 4.3.7(2) of \cite{morintao} that embeds $\hrt$ into presheaves of abelian groups on $\sm$.
Next, Proposition \ref{pdualsh} yields that objects of $\hrt$ define pure functors on $\gd$ (via $\Phi$) indeed. Besides,
by Lemma 4.3.7(2) of \cite{morintao} we get the following: it suffices to check that the restriction of an $\hrt$-representable theory $H'$ from $\shc$ to the spectra of smooth varieties yields a strictly homotopy invariant sheaf $N$ after sheafification, i.e., that $H^i_{Nis}(N)$ are homotopy invariant functors $\sm^{op}\to \ab$ (cf. Proposition \ref{psh}(\ref{ish7})). 
Loc. cit. yields that it suffices to check that  $H^i_{Nis}(N)(X)\cong H'(\om(X_+)[-i])$ for $X\in \sv$. Certainly, the latter fact would be obvious if we knew that $H'^*(\om(-_+))$ satisfies Nisnevich descent (i.e., it converts Nisnevich distinguished squares into long exact sequences; cf. part \ref{ish8} of  the proposition cited) and that for the corresponding extended $H$ and any essentially smooth Henselian scheme $Y/k$, $i\neq 0$, we have $H(\om(Y_+)[-i])=0$. The first of these statements is given by the proposition, whereas the second one is immediate from Proposition \ref{cdscoh}(3).

\end{proof}

\section{Conclusion of the proofs: the construction of $\gdp$ and $\gdb$} 
\label{sconstr}

In this section we construct the categories $\gdp$ and $\gdb$, and prove that they satisfy the 
 properties listed in Proposition \ref{pprop}.

\subsection{On   
the levelwise injective model for $\sh$: reminder}\label{srsmc}

Let us recall some properties of (certain) injective model structures for the categories mentioned in \S\ref{ssh}.
In this paper all the model categories will have functorial factorizations of morphisms.

\begin{pr}\label{ppsh}
There exist proper simplicial model structures for the categories $\dopsh$, $\dosh$,
for $\doshp$ that is equal to $\dosh$ as an 'abstract' category,
 and for a certain category $\psh$ whose homotopy category is $\sh$ satisfying the following properties.


1. The cofibrations 
in $\dopsh$ and $\dosh$ are exactly the (levelwise) injections; hence all objects of these categories are cofibrant. 

2. The natural comparison functors $\dopsh\to \dosh\to \doshp \to  \psh$ 
are left 
Quillen functors. 

3. For any $j\ge 0$ and  open embeddings $Z\to Y\to X$ the diagram corresponding to $Y/Z\brj \to X/Z\brj \to X/Y\brj$ (cf. Proposition \ref{psh}\ref{ish3})
in $\dosh$ yields a cofibration sequence.

\end{pr}
\begin{proof}
1. This is just a part of (Jardine's) definition of the injective model structures for these categories. 

2. This statement is well-known to experts; 
it can be 
deduced from the results 
of \cite{movo} 
and \cite{morintao}; see also Theorem 2.9 of \cite{jard}. 

3. For $j=0$ we obviously obtain a cofibration sequence (for the injective model structure) already in $\dopsh$ (as we have already said above). 
In the general case $Y/Z\brj \to X/Z\brj$ also yields an injection of discrete simplicial (pre)sheaves; hence it suffices to calculate the corresponding quotient sheaf. The isomorphism in question  can be easily obtained by induction using Zariski descent for sheaves (that yields that the morphism of presheaves corresponding to $X/U\times X'/U'\to X'\times X/U\times U'$ becomes an isomorphism after Zariski sheafification). 
Note also that  the natural analogue of the latter statement in $\doshp$ (that is quite sufficient for our purposes) is mentioned in \S2.3.2 of \cite{degdoc}. 


\end{proof}

\subsection{$\gdp$ and $\gdb$: definition and properties}\label{scgd} \label{sprgd}


First we define a certain (stable) model category $\gdp$ using the results of \cite{tmodel}.

As a category it will be just the category of (filtered) pro-objects of $\psh$ (see \S5 of ibid.).
 We endow it with the {\it strict} model structure; see \S5.1 of ibid. (so, weak equivalences and cofibrations are {\it essential levelwise} weak equivalence and cofibrations of pro-objects).
  An important observation here is that this model structure is a particular case of a {\it $t$-model structure} in the sense of \S6 of ibid. if one takes the following 'degenerate' $t$-structure $t'$: $\sh^{t'\ge 0}=\obj \sh$, $\sh^{t'\le 0}=\ns$. Indeed (see Remark 6.4 of ibid.), 
  $t'$ is a $t$-model structure in the sense of Definition 4.1 of ibid. (since one has a functorial factorization of any morphism $f\in \sh(X,Y)$ (for $X,Y\in \obj \sh$) as $f\circ \id_X$; note that $\id_X$ is an {\it $n$-equivalence} and $f$ is a  {\it co-$n$-equivalence} in the sense of Definition 3.2 of ibid. for any $n\in \z$; pay attention to Remark \ref{rts}(4)!). 

Now let us describe some basic properties of $\gdp$ and its homotopy category $\gdb$ (that are particular cases of the corresponding general statements). We will denote the pro-object corresponding to a projective system $X_i$ by $(X_i)$. Note that $(X_i)$ is exactly the (inverse) limit of the system $X_i$ in $\gdp$ (by the definition of morphisms in this category).

\begin{pr}\label{pgdb}

Let  $X_i,Y_i,Z_i\ i\in I$, be projective systems in $\psh$. Then the following statements are valid.

1. $\gdp$ is a proper stable simplicial model category. 

2.  
If  some morphisms $X_i\to Y_i$ for all $i\in I$ 
 yield a compatible system of cofibrations (resp. of weak equivalences; resp. some couples of morphisms 
$X_i\to Y_i\to Z_i$ yield a compatible system of cofibration sequences) then 
the corresponding 
 morphism $(X_i)\to (Y_i)$ is a cofibration also (resp. a weak equivalence; resp. the couple of morphisms $(X_i)\to (Y_i)\to (Z_i)$ is a cofibration sequence). 
 
3. The natural embedding $c:\psh\to \gdp$ is a 
left  Quillen functor; it also respects weak equivalences and fibrations.

4. For any $M\in \obj \psh$ 
we have $\gdb((X_i),c(M))\cong \inli \sh(X_i,M)$.
In particular, the 
 functor $\ho(c):\sh\to \gdb$ is a full embedding.

5. All objects of $\ho(c)(\sh)$  
are cocompact in $\gdb$. 

6. $\ho(c)(\sh)$  cogenerates $\gdb$.

\end{pr}
\begin{proof}
1. Theorems 6.3 and 6.13 of ibid. yield everything except the existence of functorial factorizations for morphisms in $\gdp$ (see of \cite{hovey}).  The existence of functorial factorizations is given by Theorem 1.2 of \cite{fufa} (see the text following Remark 1.4 of ibid.).

2. The first two parts of the assertion 
 are contained in the definition of the strict $t$-structure. The last part follows immediately (since $\gdp$ is proper). 

3. The first part of the assertion is given by Lemma 8.1 (or \S5.1) of \cite{tmodel}. 
The second part is immediate from the description of weak equivalences in $\gdp$ given in loc. cit.

4. 
Immediate from Corollary 8.7 of ibid.

5. 
We should verify that $\gdb(\prod_{i\in I} Y_i,X)=\bigoplus_{i\in I} \gdb(Y_i,X)$ for $Y_i$ being fibrant objects of  
$\gdp$, $X\in \ho(c)(\obj \sh)$. Now (see Example 1.3.11 of \cite{hovey}) the product of $Y_i$ in $\gdb$ comes from their product in $\gdp$. Certainly, if $Y_i=(Y_{ij})$ then their product in $\gdp$ 
 can be presented by the projective system of all $\prod_{i\in J} Y_{ij_i}$ for $J\subset I$ (i.e., we take at most 
  one $Y_{ij}$ for each $i\in I$ and consider the corresponding index category; note here that products are particular cases of inverse limits). Hence the statement follows 
from the previous assertion.

Alternatively, one can combine the argument dual to the one in the proof of Theorem 7.4.3 of \cite{hovey} 
with the 
fact that $c(f)$ 
for $f$ running through all fibrations in $\psh$ yield a set of 
generating fibrations for $\gdp$ (see Theorem 6.1 of \cite{chor}).

6. Loc. cit. also yields that $\gdp$ admits a non-functorial version of the generalized cosmall object argument with respect to $c(f)$. Hence we can apply the dual of the argument used in the proof of Theorem 7.3.1 of \cite{hovey}.

\end{proof}

Now let us prove Proposition \ref{pprop}. Assertions \ref{ifun} and \ref{ilim} of the proposition are given by our constructions. 
 Assertion
\ref{iemb} is given by parts 4 and 5 of the previous proposition. Assertion \ref{icoclim} also follows from part 4 of our proposition. Lastly, assertion \ref{itriang} is immediate from the combination of its part 2 with Proposition \ref{ppsh}(3).

\section{The $T$-spectral and $\mgl$-module versions of the main results; examples and  remarks}\label{ssupl}

In this section we describe 
certain 
variations of the methods and results of the previous ones, and discuss several examples of cohomology theories. We will be somewhat
sketchy sometimes.

In \S\ref{sht} we prove that the (natural analogue of the) Gersten weight structure can  be constructed for the stable motivic category of $T$-spectra also; 
the obvious analogues of the properties of $w$ established above can be proved without any difficulty.
Moreover,   the heart of the version of this weight structure for the '$\tau$-positive' part  $\shtpl$ of $\shtoh$ contains the corresponding spectra of all primitive (pro)schemes; hence in 
the corresponding pro-spectral category $\gdtpl$ all the nice 'direct summand' properties of semi-local pro-schemes (proved above) have their 'primitive' analogues. These results rely on the '$\tau$-positive' acyclity of primitive schemes that we prove in \S\ref{sshinvp}.

In \S\ref{smgl} we verify that all our results can be carried over to the triangulated category of modules over the motivic cobordism spectrum (and the corresponding pro-category). $\ho(\mglmod)$ also supports a certain {\it Chow} weight structure (if $\cha k=0$).

In \S\ref{sdm} we briefly compare our methods with the ones of \cite{bger}, and note that the results of ibid. are also valid in the case where $k$ is not countable.

In \S\ref{sexamp} we discuss the consequences of our results for 'concrete' cohomology theories (algebraic, topological, and Hermitian $K$-theory, Balmer's Witt groups, singular, algebraic and complex cobordism, \'etale and motivic cohomology).

In \S\ref{sothergd} we mention certain alternative 
methods for constructing $\gdp$ and $\gd$.

\subsection{
The $T$-spectral 
Gersten weight structures}\label{sht}

In \S5 of \cite{morintao}   the stable model category of $\p^1$-motivic spectra was considered. By Remark 5.1.10 of ibid., this category is naturally  Quillen equivalent 
to the (similarly defined) model category of $T$-spectra, where $T$ corresponds to $\pt\lan 1\ra$. We denote the latter category by $\psht$; its homotopy category will be denoted by $\sht$. Our constructions and result can be carried over from $\psh$, $\sh$ (and other (pro)spectral categories) to this setting; we will say more on this in Theorem \ref{tshtt} below. In order to prove the theorem we start from formulating the $T$-spectral analogue of Propositions \ref{psh} and \ref{psht}. 

 To this end let us recall that $\psht$ is (also) equipped with a left Quillen functor from $\doshp$; hence one can define the natural analogues $\pomt$ and $\omt$ of $\pom$ and $\om$, respectively. Besides, $\psht$ and $\sht$ are equipped with natural functors $\wedge T$. 
 The main distinction of $T$-spectra from $S^1$-ones is that $\wedge T$ is a (right) Quillen auto-equivalence of $\psht$; so we will assume it to be invertible on $\sht$.
Similarly to \S\ref{ssh}, we will denote the operation $\wedge T^j[-j]$ by $\{j\}$ (for any $j\in \z$).

Besides, $\sht$ is (also) endowed with a certain homotopy $t$-structure, which we will denote by $t^T$. It is defined (see Theorem 5.2.3 of ibid.) via the functors $\pi_n(-)_m$ for $n,m\in \z$; those send $E\in \sht$ to the Nisnevich sheafification of the presheaf $U\mapsto E^{-m}_{-n}(U)=\sht(\omt(U_+)\wedge T^{-n}[m],E)$ (see Remark 5.1.3 and Definition 5.1.12 of ibid). 
 Similarly to \S\ref{ssh} we 
  extend 
  $E^{-m}_{-n}(-)$ to all inverse limits of objects of the type $X_+$ in $\popa$.
 
 So, one easily obtains (most of) the connectivity  properties of $\sht$ listed below (paying attention to Remark \ref{rts}(4)). 
Moreover, it turns out that the class of $T$-spectra enjoying nice 'connectivity' properties is wider than the one 
coming from  semi-local pro-schemes.

 

 \begin{pr}\label{pshtt}
The following statements are valid.

\begin{enumerate}

\item\label{it4} 
The functor $-\{j\}$ is $t$-exact with respect to $t^T$ for any $j\in\z$.

\item\label{it1} For any $X\in \sv$ we have $\omt(X_+)\in \sht^{t^T\le 0}$.

\item\label{it2} For $E\in \obj \sht$ we have $E\in  \sht^{t^T\ge 0}$ if and only if $E^{n+j}_j(X_+)=0$ for all $X\in \sv$, $n<0$.

\item\label{it3} For $E\in \obj \sht$ we have $E\in  \sht^{t^T\le 0}$ if and only if $E^{n+j}_j(\spe K_{+})=0$ for all $n>0$, $j\in\z$, and any function field $K/k$. 

\item \label{ipshinvpsl} Suppose 
$k$ is infinite
 and 
 $S$ is semi-local;  then for any $E\in \sht^{t^T\le 0}$, $i> j$ we have $E^{i}_j(S_+)=\ns$.

\item \label{ipshinvp} Suppose 
 $S$ is primitive (see Definition \ref{dprim}); consider the projection $\prpl$ of $\shtoh$ (see   Remark 
\ref{rhocolim}(2))
 onto its subcategory $\shtpl$ of '$\tau$-positive' objects (see the 
proof of Theorem \ref{tshtt}(I.2) below for more details on this functor). Then for any $E\in\obj \shtoh^{t^T\le 0}$ (for the corresponding  $t^T$ for $\shtoh$; see see   Remark 
\ref{rhocolim}(2) again), $i> j$, we have $(\prpl(E))^{i}_j(S_+)=\ns$.

\end{enumerate}

\end{pr}
\begin{proof}

\ref{it4}. Immediate from the definition of $t^T$ (Definition 5.2.1 of ibid.). 

\ref{it1}. See Example 5.2.2 of ibid.

\ref{it2}. Since $t^T$ is defined in terms of the corresponding Nisnevich sheaves, we should check whether $E^{n+j}_j(S_+)=0$
for all $n<0$, $S$ being an essentially smooth Henselian scheme. Hence the assertion follows from the previous one (along  with the orthogonality axiom of $t$-structures).

\ref{it3} Applying the argument used in the proof of the previous assertion we obtain the following fact: it suffices to verify for a sheaf of the type $\pi_n(E)_m$ that it is $0$ if it vanishes at all function fields over $k$. Now, all  $\pi_n(E)_m$ are {\it strictly 
homotopy invariant} (see Remark 5.1.13 of ibid.). Hence the statement follows from Lemma 3.3.6 of ibid.

\ref{ipshinvpsl}. The statement 
follows from the previous assertions via the 
method used in the proof of Proposition \ref{pshinv}.

\ref{ipshinvp}. See the next subsection. 
\end{proof}

 Now note that minor modifications of the methods 
 used 
 for the study of $\sh$ yield
the following results. 
 
 \begin{theo}\label{tshtt}

 I.1. There exist a  natural analogue $\gdt$ of $\gd$ that is  closed with respect to all small products,
 and is equipped with an exact auto-equivalence $-\wedge T$ 
compatible with the $\sht$-version of this functor.
$\gdt$ canonically contains the triangulated subcategory $\shtc\subset \sht$ generated (in the sense described in \S\ref{snot}) by $\omt(X_+)\wedge T^j$ for $X\in \sv,\ j\in \z$, whereas 
  $\omt$ extends to a functor $\popa\to \gdt$ (that we also denote by $\omt$). 
 
 2.  $t^T$ 
 induces a $t$-structure $t^+$ on $\shtpl$.

 3.  The functor $\prpl$ 
 yields a natural exact projection functor $\prplgd $ from the category $ \gdt[\frac{1}{2}]$ 
onto its full triangulated subcategory  $\gdtpl$ (
see Remark \ref{rrcoeff}(1,2)). 
 
  II 
 The 
 following statements are  valid.

\begin{enumerate}
\item 
 There exists a weight structure $w^T$ on $\gdt$ that is  non-degenerate from below  such that $\gdt_{w^T\ge 0}$ is the coenvelope of $C'_T=\{\omt(X_+)[i]\{j\}\}$ for $X\in \sv,\ i\ge 0,\ j\in \z$, $\gdt_{w^T\le 0}={}^\perp (C'_T[1])$. 

\item $\{j\}$ is $w^T$-exact for any $j\in \z$.

\item $\hw^T$ is equivalent to the Karoubization of the category of all  
$\prod \omt(\spe K_{i+})\{j_i\}$ for $K_i$ running through function fields over $k$, $j_i\in \z$. 


\item If $f:U\to X$ is an open embedding of pro-schemes such that the complement is of codimension $\ge i$ in $X$ (see Remark \ref{rpgysin}), then $\omt(X/U)\in \gdt_{w^T\ge i}$.

\item\label{itpost} For a pro-scheme $X$ consider the Postnikov tower for $\omt(X_+)$ given by 
the natural $T$-spectral analogue of the construction in Proposition \ref{post}. 
Then this Postnikov tower is a weight one. 
In particular, if $X$ is of dimension $\le d$, then  $\omt(X_+)\in \gdt_{[0,d]}$.

\item\label{iext} 
For any  cohomological functor 
$H':\shtc\to \au$ ($\au$ satisfies AB5) construct its extension to a functor $H:\gdt\to \au$ via the method mentioned in Proposition \ref{pextc}. 

In particular, 
do so for the functor $\sht(-,M)$ (from $\shtc$ to $\ab$) 
for all   
 $M\in \obj \sht$. 
 Then the collection of these functors yields a nice duality $\Phi_T:(\gdt)^{op}\times \sht\to \ab$ such that $w^T\perp_{\Phi_T}t^T$.

\item For any extended cohomological functor $H:\gdt\to \au$ the weight spectral sequence $T=T(M,H)$ (for $M\in \obj \gdt$) corresponding to $w^T$ is functorial in $H$ and is $\gdt$-functorial in $M$ starting from $E_2$. In the case where $M=\omt(X_+)$, $X\in \sv$, one can choose $T(H,M)$ to be the 'standard' coniveau spectral sequence (starting from $E_1$; certainly, $T(H,M)$ does not depend on any choices starting from $E_2$; cf. \S\ref{sext}).

We will call such a $T(H,M)$ a {\it generalized coniveau spectral sequence} (as we also did in above).

 \item\label{itsht} For any $r\in \z$ we have $\obj\shtc\cap \sht^{t^T\le -r}=\obj\shtc\cap \gd_{w^T\ge r}$.
 
 Besides, for any $M'\in \obj\shtc \cap \sht^{t^T\le -r}$, $M=\omt(Z)$ (for a $Z\in \sv$), $g\in \gdt(M,M')$, and an extended $H:\gdt\to \au$ we have the following: any Noetherian subobject of $\imm(H(g))$ is supported at codimension $\ge r$ in $Z$ (
 cf. Proposition \ref{rwss}(II.3)). 

\item  For any $N\in \obj \sht$ the generalized coniveau spectral sequences for the functor $\Phi_T(-,N)$  can be $\gdt$-functorially expressed in terms of the $t^T$-truncations of $N$ starting from $E_2$ (cf. Proposition \ref{pdual}).

\item\label{idualt} The category of pure extended cohomological functors from $\gdt$ to an abelian $\au$ (satisfying AB5) is naturally equivalent to the category of those contravariant additive functors $\hw^T\to \au$ that convert all $\hw^T$-products into coproducts.

\item\label{isht} If $k$ is infinite then 
 the category of those contravariant additive functors $\hw^T\to \ab$ that convert all $\hw^T$-products into coproducts is also equivalent to the heart of $t^T$.

\item\label{isl} If $k$ is infinite and $X$ is semi-local,  then $\omt(X_+)
\in \gdt_{w^T=0}$.

\item\label{ids1}
 If $k$ is infinite and $S$ is a connected semi-local pro-scheme, $S_0$ is  its dense sub-pro-scheme,
then $\omt(S_+)$ is a direct summand of $\omt(S_{0+})$. 

Moreover, if $S_0=S\setminus Z$, where  $Z$ is a closed
sub-pro-scheme of $S$, then $\omt(S_{0+})\cong
\omt(S_+) \bigoplus \omt(N_{S, Z}/N_{S, Z}\setminus Z)[-1]$. If $Z$ is of codimension $j$ (everywhere) in $S$, then this decomposition takes the form
$\omt(S_{0+})\cong
\omt(S_+) \bigoplus \omt(Z_+)\wedge T^j [-1]$.
 
\item\label{ifi} Suppose 
 $k$ is infinite.
 Let $K$ be a function field  over $k$; let $K'$ be
the residue field for a geometric valuation $v$ of $K$ of rank $r$.
Then $\omt(\spe K'_+)\{r\}$ is a retract
of $\omt(\spe K_+)$.
 
\item For an infinite $k$ and an extended cohomological functor  $H$ (see assertion \ref{iext}) the natural analogues of the previous two  assertions hold (for the $H$-cohomology of the corresponding pro-schemes).

\item\label{isplc} Suppose 
 $k$ is infinite
 and 
 $S$ is a 
 semi-local pro-scheme; consider the Postnikov tower of $M=\omt (S_+)$
given by assertion \ref{itpost} 
 and denote the corresponding complex by $t(M)=M^i$. 
Denote by $T_H(S)$ the 
complex 
$(H(M^{-i}))$. Then there exist some $A^i\in \obj \au$ for $i\ge 0$ such that $T_H(S)$ is $C(\au)$-isomorphic to $H(S)\bigoplus A^0\to A^0\bigoplus A^1\to A^1\bigoplus A^2\to \dots$.

III The natural analogues of the assertions of part II hold if we replace $\gdt$ by $\gdtpl$, $\omt$ by $\omtpl=\prplgd\circ\omt[\frac{1}{2}]$, $C'_T$ by $C_+=\prplgd(C'_T)$, $w^T$ by the corresponding $w^+$, $\shtc$ by $\shctpl=\prpl(\shtc[\frac{1}{2}])$, $\Phi_T$ by  
 the corresponding $\Phi_+$, and $t^T$ by $t^+$.

Moreover, in the analogues of assertions 
\ref{isht}--\ref{isplc} it is not necessary to assume that $k$ is infinite; if it is infinite, then it suffices to assume that $S$ is primitive (in all of these assertions expect \ref{isht} and \ref{ifi}).

\end{enumerate}

\end{theo}
\begin{proof}
I.1. The main distinction of this assertion from its $\sh$-analogue is that we want to extend $\wedge T$ to an invertible exact functor on $\gdt$. This is easy since (as we have already said) $\wedge T$ possesses a 'nice lift' to $\psht$. 

2. Recall that $\prpl$ is 
 the natural projection of  
the category 
$\shtoh$ 
onto the first summand in the  Morel's decomposition 
$\shtoh\cong \shtpl\bigoplus\shtmi$ (see the text preceding Lemma 6.7 of \cite{levconv}). 
So, $\prpl$ is an idempotent endofunctor of $\shtoh$, and 
for any $M\in \obj \sht$ the object $\prpl(M[2\ob])$ is a canonical direct summand of $M[2\ob]$.
The assertion follows immediately.

3. 
Remarks \ref{rrcoeff}(1,2) and \ref{rhocolim}(2) yield the result immediately.  

II Again, we can use the $\sht$-analogues of the methods used in the previous sections (along  with Proposition \ref{pshtt}). The main distinction is given by the $-\brjj$-stability of the corresponding $C'_T$ (for any $j\in \z$); since $-\brjj$ is an automorphism of $\gdt$, 
it is $w^T$-exact.

III    Again, Remarks \ref{rrcoeff}(1,2) and \ref{rhocolim}(2) yield the  result. 
 Indeed,   all the distinguished triangles,   cocompact objects, and morphism group calculations we require 'come from' $\sht$ and $\gdt$, whereas the duality $\Phi[2\ob]:\gdt[2\ob]\times \shtoh\to \ab$ induced by $\Phi$ restricts to a duality $\Phi^+:\gdtpl\times \shtpl\to \ab$. Note here that $\Phi^+$ can be used (in particular) for the calculation of 
$\gdtpl(\omtpl(M),N)$ for any $M\in \obj\popa$, $N\in \obj \shctpl\subset \obj \gdtpl$.

\end{proof}

 \begin{rema}\label{rshtpl}
 
1. We could have considered some {\it bispectral} model for $\sht$; this would have given a natural left Quillen connection functor $\psh\to \psht$ that would certainly yield exact functors $\sh\to \sht$ and $\gd\to \gdt$. This would have allowed carrying over some of the properties from $S^1$-(pro)spectra to $T$-ones directly. 

We will also 
apply this observation in \S\ref{sexamp} below.

2. Note that $\shtpl=\shtoh$ if $-1$ is a sum of squares in $k$. Indeed, in the case $\cha k>0$ this fact is given by Lemma 6.8 of \cite{levconv}. For $\cha k=0$ this statement can be easily extracted from the proof of Lemma 6.7 of ibid.; we will give 
the detail in the next subsection.

3. We conjecture that in the case where $-1$ is a sum of squares in $k$, the natural analogues of Proposition \ref{pshtt}(\ref{ipshinvp}) an (thus) of Theorem \ref{tshtt}(III) hold for $\sht$ instead of $\shtpl$.

\end{rema}

\subsection{
The '$\tau$-positive acyclity' of primitive schemes}\label{sshinvp}


Now let us prove Proposition \ref{pshtt}(\ref{ipshinvp}).
We reduce it to the 'weakly orientable spectral acyclity of primitive schemes'.

We recall that the heart of $t^T$ is the category of {\it homotopy modules} (see Definition 5.2.4 of \cite{morintao} or Definition 1.2.2 of \cite{degorient}). This category possesses an exact faithful forgetful functor to the abelian category of 
$\z$-graded Nisnevich sheaves. It sends $M\in \sht^{t^T}=0$ into the corresponding sequence $(M_n)$; here $M_n=\pi_0(M)_n$ in the notation considered in the beginning of the previous subsection. Besides, any object $M$ of $\hrtt$ is  (functorially) equipped with a 
 system of morphisms $\eta=(\eta_n:\, M_n\to M_{n-1})$ (for $n$ running over $\z$); 
see \S6.2 of \cite{morintao} and Definition 1.2.7 of \cite{degorient}. A homotopy module is called {\it orientable} if all of these $\eta_n$ are zero. 

 Certainly, any object of $\hrtpl$ also yields a homotopy module (that is $\zoh$-linear).
 
 Now let us prove an interesting lemma that we will apply below in the case where $S$ is primitive.

 \begin{pr}\label{pacycl}
Assume that a pro-scheme $S$ possesses the following property: for any $F$ being a homotopy invariant (Nisnevich) sheaf with transfers in the sense of \S3.1 of \cite{1} we have $H^i_{Nis}(S,F)=\ns$ for all $i>0$. Then the following statements are valid.

1. Suppose 
$E $ belongs to $ \sht^{t^T\le 0}$ (resp. to $\shtpl^{t^+\le 0}$); assume also that $E^{t^T=m}$ (resp. $E^{t^+=m}$) is orientable for any $m<0$ (i.e., that $E$ is {\it weakly orientable} in the sense of 
\S4.2 of \cite{degorient}). Then for any $i> j\in \z$ we have $E^{i}_j(S_+)=\ns$. 

2. In particular, this is true for any $E\in \shtpl^{t^+\le 0}$.



\end{pr}
\begin{proof} Obviously, we can assume that $S$ is connected of some dimension $d$. Then $E^{i}_j(S_+)=\ns$ for any $E\in \sht^{t^T\le -d-1}$
 (resp. for $E\in \shtpl^{t^+\le -d-1}$;  
 immediate from Theorem \ref{tshtt}(II.\ref{itpost},\ref{iext})). 
Hence the natural spectral sequence $(E^{t^T=q})^{j+p}_j(S_+) \implies E^{j}_j(S_+)$ (resp. its $t^+$-analogue) 
 converges. Thus 
we can assume that $E=M[r]$ for some $r\ge 0$, $M\in \obj\hrtt$ (resp. in $M\in \obj \hrtpl$).

1. By Remark 1.2.4 of \cite{degorient} for $M=(M_l)$ we have $E^{i}_j(S_+)=H^{r+i-j}(S,M_j)$; here we define $H^*(-,M_j)$ on pro-schemes 
via the method of \S\ref{psh}. 
Next,
Theorem 1.3.4 of ibid. implies that $M_j$ is a homotopy invariant sheaf with transfers, and we obtain the result.

2. Certainly, it suffices to verify that all homotopy modules for $\hrtpl$ are orientable. Now 
recall  (see \S6 of \cite{levconv}) 
that the graded ring $K^{MW}_*(k)$ (functorially) acts on $\bigoplus M_j$, whereas 
for  any $M\in \obj \shtpl$ (by definition) the corresponding action of
 $\tau =\id+\eta\circ [-1]$ 
 is identical. 
Rewriting the relation 4 in Definition 6.3.1 of \cite{morintao} as $\eta \tau+ \eta =0$ 
we obtain the result.

\end{proof}

Next let us  finish the proof of Remark \ref{rshtpl}(2).  Since $-1$ is a sum of squares, it is well known that the Witt ring $W(k)$ is annihilated by $2^N$ for some $N\ge 0$. It follows that $2^N\eta=0$ in $K^{MW}_*(k)$. Indeed, this statement is given by Lemma 6.7 of \cite{levconv} in the case $\cha k>0$, whereas 
for $\cha k=0$ one should apply the argument used in the proof  of loc. cit. (i.e., use the 
natural injection of groups $W(k)\eta\to K_*^{MW}(k)$).

It remains to verify the  'motivic acyclity' of primitive schemes. 

\begin{pr} \label{pwalk}
Let $S$ be a connected primitive pro-scheme, let $S_0$ be 
 its generic
point; assume that $F:\sm\to \ab$ is a homotopy invariant Nisnevich sheaf with transfers (in the sense of \cite{1}). 

Then $H^i(F,S)=\ns$ for any $i> 0$. 
\end{pr}
\begin{proof}
In the case of an infinite $k$ 
our assertion follows from Theorem 4.19 of
\cite{walker}. 
If $k$ is finite, then $S_0$ is semi-local (by our convention); so we may apply Corollary 4.18 of \cite{3} instead.

\end{proof}


\begin{rema}
1. So, we proved the statement in question via 'decomposing' $E\in \obj \sht$ into its $t^T$-components, and using its weak orientability (in the sense of \cite{degorient}). If one wants to study the $2$-torsion part of $\sht$ (also), it could make sense to look for a filtration of $E\in \obj\sht$ such that its 'factors' (in the sense of Postnikov towers) are weakly orientable.

There is an  important and well-known candidate for such a filtration (for {\it effective} objects of $\sht$); this  is the {\it slice filtration }  
(as defined by Voevodsky). Note here that the {\it slices} are modules over the Eilenberg-Maclane spectrum $\mz$ (that represents motivic cohomology; see Theorem 3.2 of \cite{pela}); hence they are weakly orientable by Lemma 4.2.2 of \cite{degorient}.
Besides, passing to slices preserves $\sht^{t\le 0}$; see Lemma 4.3 of \cite{levadams}.
 The main problem with this method is the convergence question for the corresponding spectral sequence; it is only known (by Theorem 4 of \cite{levconv})  if $k$ is of finite cohomological dimension. 
This is rather restrictive; yet possibly the proof of loc. cit. (for $2$-torsion spectra) could be modified in order to yield the result when $-1$ is a sum of squares in $k$ (at least, when $\cha k\neq 2$).
 
2. One could try to apply the 'slice method' in order to extend %
Theorem \ref{tshtt}(III) 
to $\gd$ (or to 
$\gd[\frac{1}{p}]$ if $p>0$;  hence also  
 to cohomology theories that factor through $\shc$ or through $\shc[\frac{1}{p}]$, respectively) by  applying 
 the results of 
\cite{leslitran}. 
Yet the author does not know of any conditions that would ensure 
the convergence of the slice spectral sequence for elements of $\shc$.

3. It is a very interesting question whether one can prove (some version of) the $\sht$-acyclity of primitive schemes 'directly', i.e., using a version of Voevodsky's split standard triple argument (as in the proof of Theorem 4.19 of \cite{walker}). Possibly this can be done using Voevodsky's framed correspondences (see \cite{garpan}).

\end{rema}


\subsection{On the Gersten and Chow weight structures for modules over the algebraic cobordism spectrum}\label{smgl}

As we have seen in the previous subsection, passing from $\sht$ to $\shtpl$ enforces certain 'orientability' on $T$-spectra. Another (probably, more well-known) method for introducing 'orientations' is to consider modules over the $T$-ring spectrum $\mgl$ 
representing the motivic cobordism spectrum (see  Theorem 1.0.1 of \cite{papico}). 
Recall that $K$-theory, motivic and \'etale cohomology, and (certainly) algebraic cobordism are orientable in this sense; so that the spectra representing them are $\mgl$-modules and we can apply the results of this subsection  to these cohomology theories (see \S\ref{sexamp} below for some more details).

Now let us verify that our results can be applied to the corresponding category $\mglmod$ (of left $\mgl$-modules in $\psht$). Similarly to Proposition 38 of \cite{roe}, one can verify that 
$\mglmod$ is endowed with a proper stable model structure such the 'free module functor' $\psh\to \mglmod$  ($M\mapsto \mgl\wedge M$) is a left Quillen one. 
Actually, it seems that some effort is needed in order to guarantee the left properness of $\mglmod$; 
its right properness is obvious. We denote $\ho(\mglmod)$ by $\shmgl$.

Next,  consider the category $\proo-\mglmod$; 
the corresponding category $\gdmgl\subset \ho(\proo-\mglmod)$ will be triangulated and will contain the naturally defined  '$\mgl$-module pro-spectra' for all pro-schemes. The natural analogues of all the assertions of Proposition \ref{pprop} will be valid.

Hence in order to ensure the existence (and nice properties) of the Gersten weight structure for 
$\gdmgl$  
 (cf. Theorem \ref{tshtt}(III)) it suffices to verify the '$\mgl$-module acyclity' of primitive schemes. Certainly, the right adjoint to the free $\mgl$-module functor is the natural forgetful functor. Hence one should verify  for a primitive pro-scheme $S$ and any $X\in \sv$, $E=\omt(X_+)\wedge \mgl$, that $(E)^{i}_j(S_+)=\ns$ if $i>j\in \z$. As shown in the previous subsection, to this end it suffices to check that $E$ belongs to $\sht^{t^T\le 0}$ and is weakly orientable. The first statement is immediate from the $t$-non positivity of $\omt(X_+)$ (see Proposition \ref{pshtt}(\ref{it1})) and of $\mgl$ (see Corollary 3.9 of \cite{hoycobord}), and the fact that $\sht^{t^T\le 0}$ is $\wedge$-closed (see \S1.2.3 of \cite{degorient}). The weak orientability is given by Corollary 4.1.7 of ibid. 

\begin{rema}\label{rchowmgl}
1. One of the advantages of $\mgl$-modules is that the corresponding spectra of Thom spaces of vector bundles only depend on the bases of the bundles and on their dimension. Hence 
if $U=P\setminus \cup D_i$,
 where $P$ and all intersections of $D_i$ are products of $\gmm$, $\af^1$, projective spaces, and (the spectra of) finite extensions of $k$, then the $\mgl$-spectrum of $U$ is an {\it Artin-Tate} one, and one can write 'simple' coniveau spectral sequences for the cohomology of $U$ (see \S\ref{sat}; cf. also the proof of Theorem 4.27 of \cite{lewach} and Proposition 6.5.1 of \cite{mymot}).

2. Mark Hoyis has kindly written down the proof the following statement (see \cite{hoy}): if $\cha k=0$, then $\mgl^{2n+i}_n(X)=\ns$ for any $i\ge 0$ and smooth projective $X$. Applying the Poincare duality in $\sht$ (see Theorem 2.2 of \cite{pohu}), 
 we obtain 
the 
$\shmgl$-negativity (see Definition \ref{dwstr}) of the category of $\mgl$-module spectra of smooth projective $k$-varieties (together with their $T$-twists). Thus these twists will belong to the heart of the weight structure 'cogenerated' by them in 
 $\gdmgl$ (see Theorem \ref{tnews} and Theorem 4.5.2 (I.2) of \cite{bws}). We obtain a certain {\it Chow} weight structure on $\gdmgl$ (that could also be defined on $\shmgl$ itself; both of these structures restrict to a bounded weight structure on the subcategory of compact objects of $\shmgl$).
Recall (from Remark 2.4.3(2) and \S6.6 of \cite{bws}) that Chow weight structures yield a mighty tool for generalizing and studying weight filtrations and weight spectral sequences a-la Deligne.

Possibly, the author will treat this weight structure in a separate paper (someday).

3. We also obtain the existence of a certain Chow $t$-structure $\tcho$ on $\shmgl$ (by  Theorem 4.5.2(I.1) of \cite{bws}; cf. also \S7.1 of ibid.) such that the corresponding Chow weight structure is {\it adjacent} (i.e., orthogonal with respect to $\shmgl(-,-)$) to $w_{\chow,\shmgl}$. In particular, it follows that the heart of $\tcho$ is isomorphic to the category of additive contravariant functors from the category of 'Chow-cobordism' motives to abelian groups (see Theorem 4.5.2(II.2) of ibid). 
Besides, the Chow weight structure on $\gdmgl$ is also orthogonal to $\tcho$ with respect to the corresponding $\Phi_{\mgl}$.

4. Note that there cannot exist any Chow weight structure on $\shtc\subset\gdt$ (whose heart will contain $\omt(P_+)$ for all smooth projective $P/k$) since $\eta$ yields a non-zero morphism from $\omt(P^1_+)$ to $\omt(\pt_+)[1]$. The author does not know whether 
  the Chow weight structure can be defined on $\shtpl^c\subset \gdtpl$.  This depends on the vanishing of the $\tau$-positive parts of 
  certain '$T$-negative' stable homotopy groups of spheres over all (finitely generated) extensions of $k$. 
  On the other hand, $\shtpl[\p\ob]$ (i.e., the corresponding "rationall hull" of $\shtpl$; see Remark \ref{rhocolim}(2))  is isomorphic to the 'big' version $\dm(k)[\p\ob] $ of the category of Voevodsky's motives with rational coefficients (by a result of Morel; see Theorem 16.2.13 of \cite{degcis}); 
 hence certain Chow weight structures exist on this 
  category and also on $\gdtpl[\p\ob]$. 




\end{rema}

\subsection{Our methods and comotives}
\label{sdm} 

In \cite{bger} a certain category of {\it comotives} that we will here denote by $\gdm$ was constructed; this was a certain 'completion' of Voevodsky's $\dmge(k)$. The 'axiomatics' of $\gdm$ was quite similar to Proposition \ref{pprop} of the current paper (actually, no analogue of $\gdb$ was considered in ibid.; yet $\gdm$ satisfies all the properties needed for the application of natural analogues of the arguments above). Another distinction of ibid. from this paper is that the construction of $\gdm$  used differential graded categories (and was somewhat similar to the 'spectral version' of $\gd$ that we will mention in the next subsection); the formalism of model categories was not applied directly (yet essentially a model for $\gdm$ was considered). 

A serious disadvantage of the methods of \cite{bger} is that Gysin distinguished triangles were 
constructed using a 'purely triangulated' argument. This  
 yielded the 'comotivic' analogue of Proposition \ref{pinfgy} for countable inverse limits of smooth varieties (with respect to open embeddings) only. Also, the Gersten weight structure in ibid. was constructed via a direct verification of the negativity of the category of (Tate twists) of comotives for primitive schemes. This negativity check relied on the computation of certain higher inverse limits, that the author was only able to do for a countable $k$ (since in this case $\prli^i=0$ for $i>1$).
Also, this method  allowed to construct the Gersten weight structure only on the  subcategory of $\gdm$ 
generated by the corresponding version of $\hw$. 
Hence the corresponding computation of the heart of $w$ was 
 'automatic', whereas our one (see Theorem \ref{tgw}(\ref{iwgh}) and Theorem \ref{tnews}(III)) is 
really non-trivial (and the author was quite amazed to find out that the result is parallel to that of \cite{bger}).
Similarly, 
the verification of the fact that $w\perp_{\Phi} t$
requires more effort in our setting (see Proposition \ref{pdualsh}(II)) than  the 
proof of Proposition 4.5.1(2) of ibid.; hence the current version of the proof is 'much more 
conceptual'.
Also, it was more difficult to establish  the 'duality' between $\hw$ and $\hrt$ given by Corollary \ref{coextpure} than the corresponding comotivic fact (since strict homotopy invariance seems to be a more mysterious restriction on a Nisnevich sheaf than the existence of transfers).  Note also that we give a complete description of all pure extended cohomology theories (in Proposition \ref{pextpure}, whereas no analogue of this result was mentioned in ibid.).  

So, the  methods of the current paper can be easily applied to (various versions) of comotives (moreover, the formalism of differential graded categories gives somewhat more flexibility for this setting than the model category one, though the latter one can also be applied). 
Thus we obtain the existence of the Gersten weight structure and the 'splitting' properties for the corresponding cohomology of primitive pro-schemes (cf. Theorem \ref{tshtt}(III)) over an arbitrary perfect field $k$. 
Note also that the methods of \S6.2 of \cite{bger} work over base fields of arbitrary cardinality; so we obtain the corresponding analogue of the "Brown representability'' Corollary \ref{coextpure} for this setting. 

Besides, the differential graded formalism yield the possibility of constructing dualities of comotives with certain triangulated categories of 
'$\au$-linear motives' (with $\au\neq \ab$) 
via the connecting functor $\sht\to \dm(k)$ (the latter is the 'big' category of Voevodsky's motives; see Example 2.2.6 of \cite{degorient} or Remark \ref{rchowmgl}(4)); this yields the orthogonality of the Gersten weight structures with the corresponding 'homotopy' $t$-structures on these $\au$-motives.



Lastly,  note that  $\gdm$ supports a weight structure that extends the
Chow weight structure of $\dmge(k)$ (yet one has to invert $\cha k$  in the coefficient ring if it is positive); cf. Remark \ref{rchowmgl}(2), \S6.6 of \cite{bws}, and \cite{bzp}.

\subsection{On concrete examples of cohomology theories}\label{sexamp}

We recall that there exist natural exact connecting functors $\sh\to\sht\to \shmgl\to \dm(k)$ (see Remark \ref{rshtpl}(1)) and also $\sht\to \shtpl$. So, one may say that there are 'more' (cohomological) functors that factor through $\shc$ than those that factor through $\shtc$ or other motivic categories. Still, author does not know of many examples of cohomological functors on $\shc$ that do not factor through $\shtc$. 
Besides, if $f:\cu\to \du$ is one of the comparison functors mentioned, then the knowledge that $H^*:\cu\to \au$ factors through $\du$ yields that the corresponding generalized coniveau spectral sequences and filtrations are $\du$-functorial, which is certainly stronger than $\cu$-functoriality. So, it always makes sense to factor $H^*$ through a 'more structured' motivic category (if possible).

Thus, for cohomology theories that are $\sh$-representable one can apply the results of \S\ref{sapcoh}, whereas to $\sht$-representable theories  it makes (more) sense to apply Theorem \ref{tshtt} (see Remark \ref{rshtpl}(1)). Recall that 
 $\sht$-representable theories include 
Hermitian K-theory and Balmer's Witt groups (when $p\neq 2$; see Theorems 5.5 and 5.8 of \cite{horn}, respectively). Besides, 
  any embedding $\si:k\to \com$ yields the corresponding $\sht$-representable semi-topological $K$-theory 
  and semi-topological cobordism (see Theorem 1.0.3 of \cite{kripark}); whereas
  the natural comparison functor $Re^\si_B$ from $\sht$ to the 'topological' stable homotopy category $SH$ (see \cite{ay}) yields that all 
 'topologically stable' cohomology theories
  (including complex $K$-theory and complex cobordism) also factor through $\sht$. 
So, for all of these theories one has the $\sht$-functoriality  of (generalized) coniveau filtrations and 
spectral sequences, as well as several direct summand results for semi-local pro-schemes (if $k$ is infinite).
 
Next we note that the adjunction $\sht\leftrightarrows \shmgl$ yields that any functor $\sht$-representable by a 'strict' $\mgl$-module (i.e., by a one 
coming from $\mglmod$) also factors through $\shmgl$. 
In particular, this is the case for 
 all 
 cohomology theories that come from orientable ring spectra in $\sht$ (see Theorem 1.0.1 of \cite{papico}). 
 Hence we obtain the $\shmgl$-functoriality of the corresponding (generalized) coniveau filtrations and coniveau spectral sequences, and  
 direct summand results for primitive pro-schemes (without assuming that $k$ is infinite). Orientable cohomology theories include algebraic cobordism and algebraic $K$-theory (see Example 1.2.3 of ibid.); if we fix an embedding of $k$ into $\com$, we also obtain 'algebraic' orientability of complex $K$-theory, Morava $K$-theories, and complex cobordism (cf. Example 12.2.3(3) of \cite{degcis}). 
 
 Moreover, recall that motivic, morphic (see Theorem 5.1 of \cite{chu}),  
 etale,  singular,  and 'mixed'   (see 
 \S2.3 of \cite{hu}) cohomological functors factor through $\dm(k)$. Note here that the targets of (certain versions of) the latter three cohomology theories 
are 'richer' than $\ab$; this certainly makes our direct summand results (cf. Theorem \ref{tshtt}(II\ref{ids1})-- II\ref{isplc}, III) more interesting. 
So, we obtain 'motivic' functoriality of the corresponding filtrations and spectral sequences, and certain direct summand results for the cohomology of all primitive pro-schemes.

Lastly, note that all the triangulated motivic categories we consider are tensor ones; hence for any $Y\in \sv$ the cohomology theory $H^*_Y:X\mapsto H^*(X\times Y)$ factors through a given  motivic category whenever $H^*$ does.
Hence, we can apply all the statements established above to the corresponding 
'modifications' of the cohomology theories mentioned.

\subsection{Other possibilities for 
$\gd$}\label{sothergd}

The author doubts that all possible $\gdb$ are isomorphic (at least, if we do not add some additional 'axioms' to Proposition \ref{pprop}). Surely, there exist distinct models $\gdp$ for pro-spectra (possibly, not all of them are Quillen equivalent), and they could possess quite different model-theoretic properties. Luckily, the choices for these categories do not affect the cohomology
of pro-schemes  and objects of $\shc$ (see Remark \ref{rcohp}(2)).

We mention an alternative method for constructing $\gd$ and a model for it (avoiding $\gdb$). To this end one should
note that $\psh$ can be 
turned into a {\it spectral model category} in the sense of \cite{schwmod}. 
Next one can choose cofibrant fibrant  replacements for all isomorphism classes of objects of $\shc$ (they will form a set that we will denote  by $P$). Then Definition 3.9.1 of ibid. yields the {\it full spectral category} $E(P)$ of $P$ and (after dualizing) a functor $\homm(-,P)$ from $\psh$ to the  
{\it right spectral modules} over $E(P)$. It seems that the dual to the 
argument used in the proof of Theorem 3.9.3(i) yields the following facts: $\homm(-,P)$  is a left Quillen functor; its target is a stable model category; the restriction of $\ho(\homm(-,P))$ to $\shc$ embeds it  as a full subcategory of cocompact cogenerators into the corresponding category $\gd$. 
In particular, this methods allows to  avoid the construction of $\gdb$.
Yet 
it seems that 
this method of constructing and studying $\gd$ requires more effort than the 'pro-object' one (that we have used above).

\begin{rema}

1. Besides, the author suspects that $\gd$ can be obtained via a Bousfield localization of $\gdp$. It seems that 'classical' statements on Bousfield localizations  cannot be applied here; possibly {\it class-combinatorial} model categories could help.

2. Lastly, note 
 that our constructions of $\gdp$ and $\gdb$ can be carried over for an arbitrary $k$; the perfectness of $k$ is only required in order to ensure that residue fields of $k$-varieties are essentially smooth over $k$ (so that we can apply the Morel-Voevodsky 
  purity theorem).
\end{rema}




\end{document}